\documentclass[twoside,centertags]{amsart}

\newtheorem{innercustomthm}{Theorem}
\newenvironment{customthm}[1]
  {\renewcommand\theinnercustomthm{#1}\innercustomthm}
  {\endinnercustomthm}
  
\usepackage{fancyhdr}            
\fancypagestyle{plain}{
    \lhead{}
    \fancyhead[R]{\thepage}
    \fancyhead[L]{}
    
    \fancyfoot{}
}

\usepackage[cmtip,color,all]{xy}
\usepackage{amsmath}
\usepackage{amsfonts}
\usepackage{amssymb}
\usepackage{amsthm}
\usepackage{graphicx}
\usepackage{mathrsfs}
\usepackage{hyperref}
\usepackage{tikz-cd}
\usepackage{xcolor}
\usepackage{bbold}
\usepackage{pifont}

\newtheorem{theorem}{Theorem}[subsection]
\newtheorem{corollary}[theorem]{Corollary}
\newtheorem{lemma}[theorem]{Lemma}
\newtheorem{proposition}[theorem]{Proposition}

\theoremstyle{definition}
\newtheorem{definition}[theorem]{Definition}

\theoremstyle{definition}
\newtheorem{remark}[theorem]{Remark}
\newtheorem{example}[theorem]{Example}

\usepackage{geometry}
\usepackage{setspace}
\usepackage{stmaryrd,Commons,xcolor,soul, tikz}
\usetikzlibrary{cd}
\usepackage{circuitikz}

\usepackage[style = alphabetic, doi = false, url=false, isbn = false]{biblatex}
\usepackage{hyperref}
\usepackage{xurl}
\hypersetup{breaklinks=true}
\renewbibmacro{in:}{}
\DeclareFieldFormat[article]{title}{\mkbibemph{#1}}
\DeclareFieldFormat[misc]{title}{\mkbibemph{#1}}
\DeclareFieldFormat{journaltitle}{#1}
\definecolor{winered}{rgb}{0.5,0,0}
 \usepackage{hyperref}
 \hypersetup{citecolor = winered, linkcolor = winered, menucolor = winered, colorlinks = true}

\title{Cyclic homology of categorical coalgebras and the free loop space}
\author[M. Rivera]{Manuel Rivera}
\address{\newline M.R., Department of Mathematics, Purdue University, 150 N University St., West Lafayette, IN, USA 47907}
\email{\href{mailto:manuelr@purdue.edu}{manuelr@purdue.edu}}
\author[D. Tolosa]{Daniel Tolosa}
\address{\newline D.T., Department of Mathematics, Purdue University, 150 N University St., West Lafayette, IN, USA 47907}
\email{\href{dtolosav@purdue.edu}{dtolosav@purdue.edu}}

\subjclass[2020]{55P35, 16E40, 57T30, 18M70}

\begin{document}

\singlespacing

\maketitle
\begin{abstract}{We prove that the cyclic chain complex of the categorical coalgebra of singular chains on an arbitrary topological space $X$ is naturally quasi-isomorphic to the $S^1$-equivariant chains of the free loop space of $X$. This statement does not require any hypotheses on $X$ or on the commutative ring of coefficients. Along the way, we introduce a family of polytopes, coined as Goodwillie polytopes, that controls the combinatorics behind the relationship of the coHochschild complex of a categorical coalgebra and the Hochschild complex of its associated differential graded category.}
\end{abstract}
\section{Introduction}

Given any topological space $X$, its \textit{free loop space} $\mathcal{L}X$ is the space of all continuous maps from the circle $S^1$ to $X$. Rotation of loops endows $\mathcal{L}X$ with a continuous action of the circle group. Furthermore, evaluating loops at a fixed point on the circle gives rise to a fibration $e \colon \mathcal{L}X \to X$. The fiber $e^{-1}(b)= \Omega_bX$ over any $b \in X$, called the space of \textit{based loops at $b$}, comes with a topological monoid structure given by concatenation of loops. All of this structure is invariant under weak homotopy equivalence; namely, a weak homotopy equivalence $X \to X'$ induces a weak homotopy equivalence $\mathcal{L}X \to \mathcal{L}X'$ preserving every part of the structure. A natural question is how the algebraic topology of different parts of this structure are related. For instance, the (co)homology groups of $X$, $\mathcal{L}X$ and $\Omega_bX$ are related through the Serre spectral sequence. 

In this article, we prove that the singular chains on $\mathcal{L}X$, with coefficients in a fixed commutative ring $R$, together with its cyclic symmetry induced by rotation of loops can be derived directly, up to quasi-isomorphism, from the coassociative coalgebra structure of the normalized singular chains on $X$ suitably considered. In particular, we show the $S^1$-equivariant homology of $\mathcal{L}X$ can be determined functorially from the coalgebra of singular chains on $X$. This is achieved by considering the chains on $X$ as a \textit{categorical coalgebra}, applying a version of the \textit{coHochschild complex}, and taking its associated chain complex of \textit{cyclic chains}. This extends previous results in the literature relating the algebra of singular chains on the based loop space of a pointed space $(X,b)$ or the algebra of cochains (or differential forms) on $X$ to the free loop space via cyclic homology \cite{Goodwillie, jones1987cyclic, getzlerjones, cieliebakvolkov}. Our result assumes no connectivity hypotheses on $X$, choice of base-point, or restrictions on the commutative ring $R$. This is achieved by building upon recent results of the first author relating the cobar construction of the coalgebra of singular chains on $X$ to the algebra of singular chains on the based loop space of $X$ in the possibly non-simply connected setting. 

We describe the main results of this article in more detail. A categorical coalgebra consists of a non-negatively graded $R$-coalgebra with a set of ``objects" (in degree $0$) together with a degree $-1$ boundary map that is a coderivation of the coproduct and for which every degree $1$ element is a cycle. The boundary map is not assumed to square to zero, but its failure is controlled by a curvature term that is regarded as part of the structure (\ref{categoricalcoalgebra}). This notion is dual in a precise sense to that of a \textit{differential (non-negatively) graded $R$-category}. In fact, any categorical coalgebra $C$ gives rise to a differential graded (dg) category $\mathbf{\Omega}C$ by applying an appropriate version of the \textit{cobar functor} (\ref{cobarfunctor}). The idea is that morphisms in $\mathbf{\Omega}C$ are chain complexes generated by ordered sequences of ``composable" elements in $C$ with differential determined by the coproduct, boundary, and curvature of $C$. For any space $X$, the normalized singular chains $\mathcal{C}(X)$ with \textit{Alexander-Whitney diagonal approximation} as coproduct may be regarded as a categorical coalgebra by slightly modifiying the usual boundary map (\ref{categoricalchains}). The dg category $\mathbf{\Omega}(\mathcal{C}(X))$ is quasi-equivalent to the dg category whose objects are the points of $X$ and morphisms are the chains on the space of paths with fixed endpoints with composition induced by path concatenation (\ref{adamsthm}).

We describe how any categorical coalgebra $C$ also gives rise to a \textit{mixed complex}, denoted by $(\textnormal{coHoch}_\bullet(C), d, P)$,  through an appropriate version of the coHochschild complex (\ref{cohochschild}, \ref{Hochasmixedcomplex}). Here $d$ is a differential of degree $-1$ and $P$ a differential of degree $+1$ compatible with $d$, which we may think of as encoding cyclic symmetry. Dually, any dg category $A$ gives rise to a mixed complex $(\textnormal{Hoch}_\bullet(A), \delta, B)$ given by the classical Hochschild complex equipped with \textit{Connes' operator} $B$.  Another example of a mixed complex is $(C_\bullet(\mathcal{L}X), \partial, \mathcal{R})$, the complex of normalized singular chains on $\mathcal{L}X$ with boundary operator $\partial$ and degree $+1$ differential $\mathcal{R}$ arising from the circle action on $\mathcal{L}X$ (\ref{chainsonLX}). Furthermore, any mixed complex gives rise to different versions of cyclic chains computing either \textit{positive, negative, or periodic cyclic homology} (\ref{3versions}). 
Our result can now be stated more precisely as follows.

\begin{customthm}{1}[\ref{maintheorem}] \label{thm1}
For any topological space $X$, there is a zig-zag of quasi-isomorphisms of mixed complexes
    \begin{eqnarray}\label{zigzag}
    (\textnormal{coHoch}_\bullet (\mathcal{C}(X)), d,P) \xleftarrow{\simeq} (\textnormal{Hoch}_\bullet( \mathbf{\Omega} \mathcal{C}(X)),\delta,B) \xrightarrow{\simeq} (C_\bullet(\fancy{L}X), \partial, \mathcal{R}),
    \end{eqnarray} 
    inducing a natural isomorphism on homology. 
    As a consequence, the positive cyclic homology of the categorical coalgebra $\mathcal{C}(X)$ is naturally isomorphic to $H^{S^1}_\bullet(\mathcal{L}X;R)$, the $S^1$-equivariant homology of $\mathcal{L}X$. 
\end{customthm}
The middle object in the above zig-zag can be thought of as a ``thicker" model that has more ``space" for compatibilities to hold strictly. In the course of proving the above theorem, we discovered a combinatorial picture explaining this ``thickening". This may be traced down to the following key observation. Recall that any dg category $A$ (that is appropriately $R$-cofibrant) may be resolved as an $A$-bimodule via a version of the two-sided bar resolution $\mathcal{B}(A,A,A)$, where ``free copies of $A \otimes A^{op}$" are indexed by generating monomials of morphisms in $A$.  When $A=\mathbf{\Omega}C$, for a categorical coalgebra $C$, there is a smaller resolution $\mathcal{Q}(A,C,A)$ where free copies of $A \otimes A^{op}$ are now indexed by elements of $C$ instead of being indexed by monomials of morphisms in $\mathbf{\Omega}C$ as in the bar resolution (\ref{comparingresolutions}). When $C=\mathcal{C}(X)$ and $A=\mathbf{\Omega}C$, we explain in Theorem \ref{thmresolutions} how the natural quasi-isomorphism 
\begin{eqnarray} \label{resolutions}
    \mathcal{Q}(A,C,A) \xrightarrow{\simeq} \mathcal{B}(A,A,A),
\end{eqnarray}
which may be thought of as a manifestation of \textit{dg Koszul duality}, corresponds to a geometric map between two different models for the normalized singular chains on the space of paths in $X$. This  map is described in terms of the combinatorics of a family of polytopes (of independent interest), which we coined as \textit{Goodwillie polytopes} (because of their relation with a construction in \cite{Goodwillie}) and are closely related to \textit{freehedra} \cite{saneblidze, rivera-saneblidze,poliakova2020cellular} (\ref{gpolytopes}). As a consequence, we obtain a geometric picture for a chain homotopy inverse of the first quasi-isomorphism in \ref{zigzag} (\ref{freeloopsmodel}). Goodwillie polytopes are interesting combinatorial objects in their own sake. For example, computer calculations have lead us to conjecture that these define a family of \textit{self-dual} polytopes, so their  face-vectors are palindromic. This is a rather unusual property that suggests a deeper symmetry behind the algebraic constructions.

\subsection*{Motivation} This article is motivated by string topology of manifolds. When $X$ is a manifold, the $S^1$-equivariant homology of $\mathcal{L}X$ relative to constant loops, $H^{S^1}_\bullet(\mathcal{L}X,X)$, carries an involutive Lie bialgebra structure defined by transversally intersecting, cutting, and concatenating families of ``strings" (or $S^1$-equivariant loops) in $X$. This structure should lift to an $\textnormal{IBL}_{\infty}$-algebra (a homotopy coherent version of an involutive Lie bialgebra) at the chain level, which, based on  the results of \cite{naef,NRW}, should be a source of potentially interesting manifold invariants. \textit{Our results are motivated by constructing a small, tractable, and general model for the $S^1$-equivariant chains of $\mathcal{L}X$, based on a decomposition of $X$ into cells such as simplices or cubes, that makes transparent how geometric structures on $X$ manifest at the level of strings.} We conjecture that, when $X$ is an oriented smooth manifold, a relative version of the cyclic chains associated to the mixed complex $(\textnormal{coHoch}_\bullet (\mathcal{C}(X)), d,P)$ carries an $\textnormal{IBL}_{\infty}$-algebra structure lifting the string topology Lie bialgebra and extending the structures constructed in \cite{CHV} for simply connected closed manifolds and $R=\mathbb{R}$ to the more general setting of non-simply connected manifolds and arbitrary coefficients.

\subsection*{Organization}
We start the article by discussing in section \ref{polytopes} two families of polytopes that will appear when comparing different models for path spaces and loop spaces. This part of the story is self-contained and of independent interest. We proceed by discussing generalities about categorical coalgebras and dg categories in section \ref{categoricalcoalgebrasanddgcategories}. In section \ref{chainmodels}, we describe a geometric/combinatorial picture for \ref{resolutions} in terms of the polytopes introduced in section \ref{polytopes}. Finally, in section \ref{cyclichomology}, we prove Theorem \ref{thm1}.

\subsection*{Preliminaries} \label{preliminaries}
We denote by $\mathsf{Top}$ the category of topological spaces. We will always consider \textit{Moore paths} when defining path and loop spaces. For instance, the \textit{space of paths in} $X$ is defined as the set
$$PX = \{ (\gamma, l) | l \in [0,\infty)\text{ and } \gamma \colon [0,l] \to X \text{ is a continuous map}\}$$
equipped with the natural topology induced by the compact-open topology. The \textit{free loop space} is defined as the subspace $\mathcal{L}X=\{ (\gamma, l)\in PX | \gamma(0)=\gamma(l) \} \subset PX$. These constructions give rise to functors $\mathsf{Top} \to \mathsf{Top}$.

Throughout the article, $R$ denotes a fixed commutative ring with unit and we write $\otimes=\otimes_R$. A \textit{graded coassociative $R$-coalgebra} $(C, \Delta)$ consists of a graded $R$-module $C$ and an $R$-linear map $\Delta \colon C \to C \otimes_R C$ of degree $0$ satisfying $(\Delta \otimes_R \text{id}_C ) \circ \Delta =  (\text{id}_C \otimes_R \Delta ) \circ \Delta$. We say $(C, \Delta)$ is \textit{counital} if there is an $R$-linear map $\varepsilon \colon C \to R$ of degree $0$ such that $(\varepsilon \otimes_R \text{id}_C ) \circ \Delta = \text{id}_C = (\text{id}_C \otimes_R \varepsilon) \circ \Delta$. We denote by $C^{\text{op}}$ the graded coalgebra whose underlying $R$-module is $C$, and its coproduct is given by $\Delta^{\text{op}} = t\circ \Delta$ where $t:C\otimes C \rightarrow C\otimes C$ is defined by $t(c\otimes c')= (-1)^{|c| |c'|} c'\otimes c $. A \textit{dg} $R$-\textit{module}, or \textit{chain complex}, $(M,\partial)$ consists of a graded $R$-module $M$ together with an $R$-linear map $\partial \colon M \to M$ of degree $-1$, called a \textit{differential}, satisfying $\partial^2=0$. We denote the category of dg $R$-modules with $R$-linear maps of degree $0$ commuting with differentials (\textit{chain maps}) as morphisms by $\mathsf{Ch}_R$. A \textit{dg $R$-coalgebra} $(C, \partial, \Delta)$ consists of a graded $R$-module $C$ such that $(C,\partial)$ is a dg $R$-module, $(C,\Delta)$ is a graded coassociative coalgebra, and $\partial$ is a coderivation of $\Delta$, i.e. $(\partial \otimes \text{id}_C + \text{id}_C \otimes \partial) \circ \Delta= \Delta \circ \partial$. Any set $S$ gives rise to a cocommutative $R$-coalgebra $R[S]=(R\langle S \rangle, \Delta, \varepsilon)$, where $R\langle S \rangle$ denotes the free $R$-module generated by $S$ with coproduct and counit determined by declaring $\Delta(s)=s \otimes s$ and $\varepsilon (s)=1_R$, respectively, for any basis element $s \in S$. 

When applying maps between graded objects we follow the \textit{Koszul sign rule}.

\subsection*{Acknowledgments}
The authors would like to acknowledge support of NSF grant DMS 210554. MR would like to thank Kai Cieliebak, Julian Holstein, Samson Saneblidze,  Boris Tsygan, and Mahmoud Zeinalian for helpful discussions. The authors also thank Frederic Chapoton for a correction on a bound on section 2.2 and for sharing some computer calculations of face-vectors of Goodwillie polytopes. 

\section{Polytopes}\label{polytopes}
In this section we describe two families of polytopes. The polytopes in the first family are called \textit{freehedra}. These were introduced in \cite{saneblidze} and discussed further in \cite{rivera-saneblidze} in the context of combinatorial models for the free loop space fibration; see also \cite{poliakova2020cellular}. The polytopes in the second family will be called \textit{Goodwillie polytopes} and, to our knowledge, the present article gives a detailed account of these for the first time. There is an evident connection between freehedra and Goodwillie polytopes that may be understood as a combinatorial picture behind \textit{differential graded Koszul duality}. 

\subsection{Freehedra}\label{def:freehedra}
Let $I^n$ be the $n$-dimensional unit cube in $\mathbb{R}^n$ and denote by 
\[V_{I^n}=\{ (a_1, \dots , a_n) \in \mathbb{R}^n \,\vert \, a_i \in \{0,1\}\}\] its set of vertices. Subdivide each of the $n-1$ faces of the form
\[\set{ (a_1, \dots , a_n) \in I^n \, \vert a_i =0},\]
for $i=2, \ldots, n$,
into two new faces along the hyperplane determined by $a_i=0$ and $a_1 = {(i-1)/n}$. This subdivision of the cube determines a cell complex, denoted by $F_n$, embedded inside $\mathbb{R}^n$. We call $F_n$ the \textit{$n$-freehedron}. We call  \textit{facets} the codimension $1$ faces of a polytope or cell complex.

Note that $F_n$ has $3n-1$ facets. We label each of these as follows:
\begin{align*}
    C_{0,i,n}^{'} &:= \set{(a_1, \dots , a_n) \in I^n \, \vert a_i =0 , \text{ and } a_1 \geq {(i-1)/n} } \text{ for }i= 2,\dots, n \\
    C_{0,i,n}^{''} &:= \set{(a_1, \dots , a_n) \in I^n \, \vert a_i =0 , \text{ and } a_1 \leq {(i-1)/n} } \text{ for }i= 2,\dots, n\\
    C_{0,1,n} &:= \set{(a_1, \dots , a_n) \in I^n \, \vert a_1 =0}\\
    C_{1,i,n} &:= \set{(a_1, \dots , a_n) \in I^n \, \vert a_i =1}.
\end{align*}
The codimension $2$ faces of the freehedron are determined by the non-empty intersections of the facets, codimension $3$ faces are determined by the non-empty triple intersections of the facets, and so on. This provides a labeling of all faces of the freehedron (up to rearrangement of terms in the intersection). This data determines an abstract polytope with poset structure given by face inclusions. We denote this abstract polytope by $\fancy{F}_n$, and also call it the $n$-freehedron.

We recall an alternate description of $\fancy{F}_n$ following \cite{poliakova2020cellular}, but using slightly different language and notation. In this description, faces are labeled by what we call \textit{$n$-strings} and face inclusions are determined by \textit{face transformations} as we now define.
\begin{definition} \label{strings}
\normalfont  Let $n$ be a non-negative integer.  An \textit{$n$-string} is an ordered sequence 
    \[s = \{ s_0 | \cdots |s_{l-1} \} s_l \{ s_{l+1} | \cdots | s_k \}\]
    such that
    \begin{itemize}
    \item $k$ and $l$ are integers with $0 \leq k \leq n$ and $0 \leq l \leq k$
   \item each $s_i$ is a nonempty subset of $\{0,1,\dots, n\}$ 
        \item for each $i=0,\ldots,k-1$, we have $\max s_i = \min s_{i+1}$
        \item $\abs{s_i}\geq 2$ if $i\neq l$ ($\abs{s_l}=1$ is allowed)
        \item $\min s_0 =0$ and $\max s_k =n$
      
    \end{itemize}
    In the above definition, if $l=0$ we write $s=\{ \, \} s_0\{s_1| \cdots |s_k\}$ and if  $l=k$ we write $s=\{s_0| \cdots |s_{k-1}\}s_k\{ \, \}$.  If $k=l=0$, we write $s=\{ \, \}s_0\{ \, \}$. We will omit the curly brackets in the set notation when explicitly describing an $s_i$ to avoid confusion, e.g. we denote the $4$-string $\{ \{0,1\} |\{1,3\}\} \{3\} \{ \{3,4\} \}$ simply by $\{0,1 | 1,3 \}3 \{3,4\}.$

 A \textit{face transformation} on $s = \{ s_0 | \cdots |s_{l-1} \} s_l \{ s_{l+1} | \cdots | s_k \}$ is defined to be one of the following operations:
    \begin{enumerate}
        \item Drop: for some $s_j$ in $s$ remove an element $x \in s_j$ with $\min s_j < x< \max s_j.$
        \item Inner break: replace some $s_j$, $j \neq l$, with $s_j^1|s_j^2$ where $s_j^1 = \{ a\in s_j \vert a \leq x \}$, $s_j^2 = \{ a\in s_j \vert a \geq x \}$ for some $x \in s_j$ with $\min s_j < x < \max s_l$, and neither $s_j^1$ or $s_j^2$ are singletons. 
        \item Right outer break: replace \[s = \{ s_0 | \cdots |s_{l-1} \} s_l \{ s_{l+1} | \cdots | s_k \}\] with \[s = \{ s_0 | \cdots |s_{l-1} \} s_l^1 \{ s_{l}^2 | s_{l+1} | \cdots | s_k \}\] where $s_l^1 = \{a\in s_l \vert  a \leq x \}$, $s_l^2 = \{  a\in s_l \vert a \geq x\}$ for some $x \in s_l$ with $x < \max s_l$, and the set $s^1_l$ (but not $s^2_l$) is allowed to be a singleton.
        \item Left outer break: replace \[s = \{ s_0 | \cdots |s_{l-1} \} s_l \{ s_{l+1} | \cdots | s_k \}\] with \[s = \{ s_0 | \cdots |s_{l-1} | s_l^1 \} s_{l}^2 \{ s_{l+1} | \cdots | s_k \}\] 
        where $s_l^1 = \{a\in s_l \vert  a \leq x \}$, $s_l^2 = \{  a\in s_l \vert a \geq x\}$ for some $x \in s_l$ with $\min s_l<x < \max s_l$, and the set $s^2_l$ (but not $s^1_l$) is allowed to be a singleton.
    \end{enumerate}
\end{definition}

We label the top dimensional $n$-cell of $F_n$ by the $n$-string $\{ \, \}0,\ldots ,n\{ \, \}$, (so that $k=l=0$ and $s_0$ is the set $\{0,1,\ldots,n\}$). To each $(n-1)$-dimensional face of $F_n$ we assign an $n$-string in the following way:
\begin{align}\label{faces_and_breaks}
    C_{0,i,n}^{'} &\longmapsto \{0 , \ldots, i-1 \}i-1, \dots, n \{ \, \} \text{ for }i= 2,\dots, n \\
    C_{0,i,n}^{''} &\longmapsto \{ \, \}0,\ldots,i-1\{i-1, \dots, n\} \text{ for }i= 2,\dots, n \\
    C_{0,1,n} &\longmapsto \{ \, \} 0\{0, \dots, n\} \\
    C_{1,1,n} &\longmapsto \{0, \ldots, n\}n\{ \, \} \\
    C_{1,i,n} &\longmapsto \{ \, \}0,\dots,\widehat{i-1},\dots ,  n\{ \, \}   \text{ for } i=2,\dots , n
\end{align}
This extends to faces of all dimensions through the following correspondence between the operations of intersection and face transformations:
\begin{align*}
    \cap C_{0,i,n}^{'} &\longmapsto \text{Left outer break at }x= i-1 \text{ if $i-1 \in s_l$, inner break otherwise}\\
    \cap C_{0,i,n}^{''} &\longmapsto \text{Right outer break at }x= i-1 \text{ if $i-1 \in s_l$, inner break otherwise} \\
    \cap C_{0,1,n} &\longmapsto \text{Right outer break at }\min s_l\\
    \cap C_{1,1,n} &\longmapsto \text{Left outer break at }\max s_l\\
    \cap C_{1,i,n} &\longmapsto \text{Drop } i-1 \\
\end{align*}
Under this correspondence, a face $s'$ is a codimension 1 face of $s$ if it can be obtained from $s$ via a face transformation. This assignment gives rise to an alternate description of $\fancy{F}_n$ in terms of $n$-strings and face transformations. See figure \ref{fig:F2-with-labels} below for an example when $n=2$.
\begin{figure}[h]
\centering
\resizebox{.45\textwidth}{!}{%
\begin{circuitikz}
\tikzstyle{every node}=[font=\normalsize]
\draw [](5,11.75) to[short] (5,6.75); 
\draw [](5,6.75) to[short] (10,6.75);
\draw [](10,11.75) to[short] (10,6.75);
\draw [](5,11.75) to[short] (10,11.75);
\draw (5,11.75) to[short, -*] (5,11.75); 
\draw (5,6.75) to[short, -*] (5,6.75); 
\draw (10,6.75) to[short, -*] (10,6.75); 
\draw (10,11.75) to[short, -*] (10,11.75);
\draw (7.5,6.75) to[short, -*] (7.5,6.75); 
\node [font=\normalsize] at (7.5,9.25) {};
\node [font=\normalsize] at (7.5,12.05) {$C_{1,2,2}$};
\node [font=\normalsize] at (10.5,9.25) {$C_{1,1,2}$};
\node [font=\normalsize] at (4.5,9.25) {$C_{0,1,2}$};
\node [font=\normalsize] at (6.25,6.4) {$C_{0,2,2}^{''}$};
\node [font=\normalsize] at (8.75,6.4) {$C_{0,2,2}^{'}$};
\draw [->, >=Stealth] (4.35,12.4) to (4.85,11.90); 
\draw [->, >=Stealth] (10.65,12.4) to (10.15,11.9); 
\draw [->, >=Stealth] (10.65,6.1) to (10.15,6.6); 
\draw [->, >=Stealth] (4.35,6.1) to (4.85,6.60); 
\draw [->, >=Stealth] (7.5,6.1) to (7.5,6.6); 
\node [font=\normalsize] at (4,12.75) {$C_{1,2,2} \cap C_{0,1,2}$};
\node [font=\normalsize] at (11,12.75) {$C_{1,1,2} \cap C_{1,2,2}$};
\node [font=\normalsize] at (4,5.75) {$C_{0,1,2} \cap C^{''}_{0,2,2}$};
\node [font=\normalsize] at (11,5.75) {$C_{1,1,2} \cap C^{'}_{0,2,2}$};
\node [font=\normalsize] at (7.5,5.75) {$C^{'}_{0,2,2}\cap C^{''}_{0,2,2}$};
\end{circuitikz}
}%
\resizebox{.45\textwidth}{!}{%
\begin{circuitikz}
\tikzstyle{every node}=[font=\normalsize]
\draw [](5,11.75) to[short] (5,6.75); 
\draw [](5,6.75) to[short] (10,6.75);
\draw [](10,11.75) to[short] (10,6.75);
\draw [](5,11.75) to[short] (10,11.75);
\draw (5,11.75) to[short, -*] (5,11.75); 
\draw (5,6.75) to[short, -*] (5,6.75); 
\draw (10,6.75) to[short, -*] (10,6.75); 
\draw (10,11.75) to[short, -*] (10,11.75);
\draw (7.5,6.75) to[short, -*] (7.5,6.75); 
\node [font=\small] at (7.5,9.25) {$\{\} 0,1,2 \{\}$}; 
\node [font=\small] at (7.5,12.05) {$\{\}0,2\{\}$}; 
\node [font=\small] at (11.15,9.25) {$\{0,1,2\}2\{ \}$}; 
\node [font=\small] at (3.8,9.25) {$\{\} 0 \{ 0,1,2 \}$}; 
\node [font=\small] at (6.25,6.4) {$\{\} 0,1 \{ 1,2 \} $}; 
\node [font=\small] at (8.75,6.4) {$\{ 0,1 \}1,2\{ \}$};
\node [font=\small] at (4,12.75) {$\{\}0\{0,2\}$}; 
\node [font=\small] at (11,12.75) {$\{0,2\}2\{ \}$}; 
\node [font=\small] at (4,5.75) {$\{ \}0 \{ 0 ,1 | 1,2 \}$}; 
\node [font=\small] at (11,5.75) {$\{ 0,1 | 1,2 \}2\{\}$}; 
\node [font=\small] at (7.5,5.75) {$\{0,1 \} 1 \{ 1,2 \}$}; 
\draw [->, >=Stealth] (4.35,12.4) to (4.85,11.90); 
\draw [->, >=Stealth] (10.65,12.4) to (10.15,11.9); 
\draw [->, >=Stealth] (10.65,6.1) to (10.15,6.6); 
\draw [->, >=Stealth] (4.35,6.1) to (4.85,6.60); 
\draw [->, >=Stealth] (7.5,6.1) to (7.5,6.6); 
\end{circuitikz}
}%
\caption{Two labelings of the freehedron $F_2$.}
\label{fig:F2-with-labels}
\end{figure}
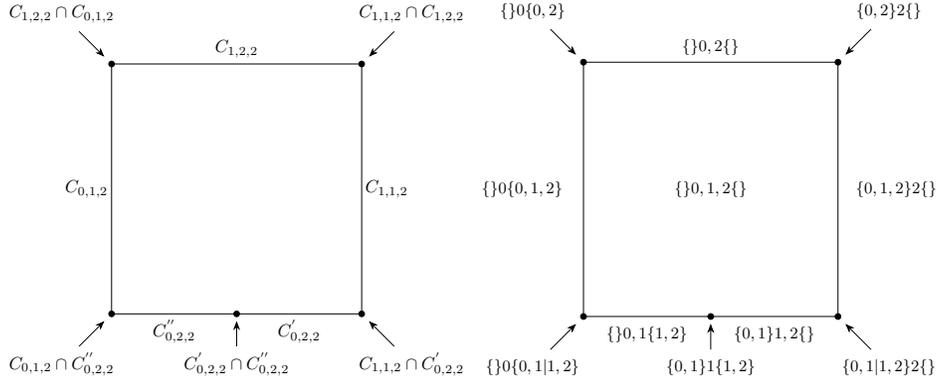


\subsection{Goodwillie polytopes}\label{gpolytopes}
Let $n>2$ be an integer and $\epsilon \in [0,\frac{1}{n-1}]$. Consider the set of points 
\[
U_{G_n^\epsilon} =  \left\{ \left(\frac{i-1}{n}, a_2,\dots, a_n \right) \in \mathbb{R}^n \,\vert \, i\in \{ 2,\dots,n \}, a_i = -1, a_j \in \{ \epsilon , 1-\epsilon\} \text{ for } j \neq i \right\}.
\]
Denote $V_{G^{\epsilon}_n} :=V_{I^n} \cup U_{G_n^\epsilon}$, where $V_{I^n}$ is the set of vertices of the standard $n$-cube, and define the polytope $G_n^\epsilon$ as the convex hull of the set $V_{G^{\epsilon}_n}$. If $ 0< \epsilon, \epsilon '< \frac{1}{n-1}$, then $G_n^{\epsilon}$ and $G_n^{\epsilon'}$ are combinatorially equivalent. In this case, we denote the underlying abstract polytope by $\fancy{G}_n$ and call it the $n$-th \textit{Goodwillie polytope}. Furthermore, note that $G_n^{0}$, $G_n^{\frac{1}{n}}$, and $G_n^{\frac{1}{n-1}}$ may be pairwise combinatorially distinct, see Figure \ref{decompfigure}. 
\begin{figure}[h]
    \centering
    \includegraphics[width= .8\textwidth]{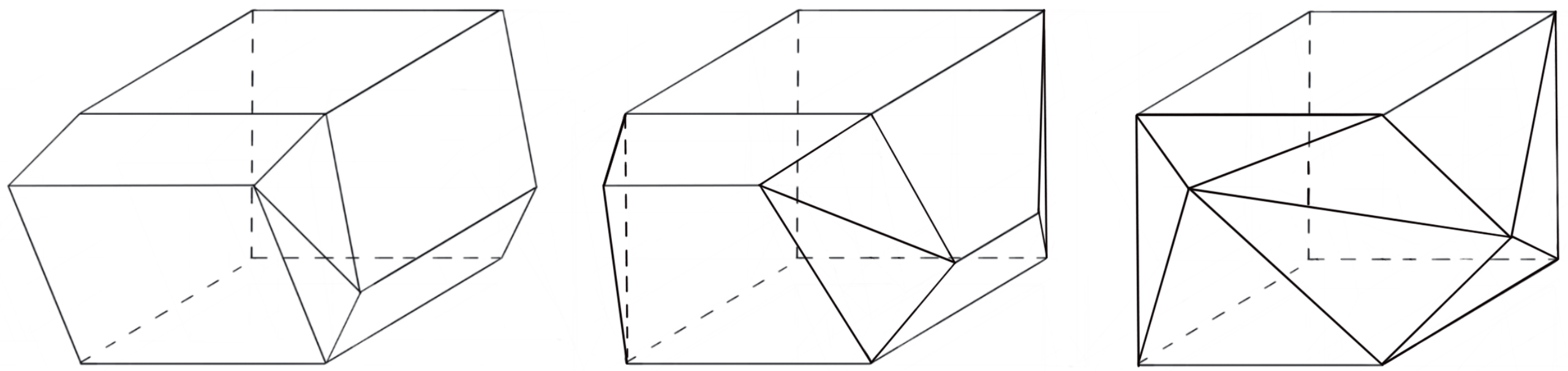}
    \caption{(left) $G_3^{0}$, (center) $G_3^{\epsilon}$, where $0< \epsilon < \frac{1}{2}$, and (right) $G_3^{\frac{1 }{  2}}$.}
    \label{fig:G3_epsilon}
\end{figure}

If  $0 \leq \epsilon < \frac{1}{n-1}$, then $|U_{G^\epsilon_n}|= (n-1)2^{n-2}$. Hence, $\fancy{G}_n$ has $|V_{G^{\epsilon}_n}|= |V_{I^n} \cup U_{G^{\epsilon}_n}|=2^n +(n-1)2^{n-2}=2^{n-2}(3+n)$ vertices, just like $\fancy{F}_n$.

\begin{figure}[!ht]
\centering
\resizebox{.7\textwidth}{!}{%
\begin{circuitikz}
\tikzstyle{every node}=[font=\LARGE]
\draw [short] (2.5,15.5) .. controls (2.5,10) and (2.5,10) .. (2.5,4.25); 
\draw [short] (0,9.25) .. controls (10,9.25) and (10,9.25) .. (20,9.25);
\draw [short] (0,14.25) .. controls (10,14.25) and (10,14.25) .. (20,14.25);
\draw [short] (6.25,15.5) .. controls (6.25,10) and (6.25,10) .. (6.25,4.25);
\draw [short] (10,15.5) .. controls (10,10) and (10,10) .. (10,4.25);
\draw [short] (15,15.5) .. controls (15,10) and (15,10) .. (15,4.25);
\draw [short] (20,15.5) .. controls (20,10) and (20,10) .. (20,4.25);
\draw [short] (0,15.5) .. controls (0,10) and (0,10) .. (0,4.25);
\draw [short] (0,4.25) .. controls (10,4.25) and (10,4.25) .. (20,4.25);
\draw [short] (0,15.5) .. controls (10,15.5) and (10,15.5) .. (20,15.5);
\node [font=\LARGE] at (1.25,11.8) {$F_n$};
\node [font=\LARGE] at (1.25,6.7) {$G_n^\epsilon$};
\node [font=\LARGE] at (4.4,14.9) {0};
\node [font=\LARGE] at (8.1,14.9) {1};
\node [font=\LARGE] at (12.5,14.9) {2};
\node [font=\LARGE] at (17.5,14.9) {3};
\draw (4.5,11.75) to[short, -*] (4.5,11.75); 
\draw (4.5,6.75) to[short, -*] (4.5,6.75); 
    \draw [line width=1pt, short] (7.5,11.75) .. controls (8.25,11.75) and (8.25,11.75) .. (8.75,11.75);
    \draw (7.5,11.75) to[short, -*] (7.5,11.75);
    \draw (8.75,11.75) to[short, -*] (8.75,11.75);
\draw [line width=1pt, short] (7.5,6.75) .. controls (8.25,6.75) and (8.25,6.75) .. (8.75,6.75);
\draw (7.5,6.75) to[short, -*] (7.5,6.75);
\draw (8.75,6.75) to[short, -*] (8.75,6.75);
\draw [line width=1pt, short] (11.25,13) to (11.25,10.5); 
\draw [line width=1pt, short] (11.25,10.5) to (13.75,10.5); 
\draw [line width=1pt, short] (13.75,13) to (13.75,10.5); 
\draw [line width=1pt, short] (11.25,13) .. controls (12.5,13) and (12.5,13) .. (13.75,13); 
\draw (12.5,13) to[short, -*] (12.5,13); 
\draw (11.25,13) to[short, -*] (11.25,13); 
\draw (13.75,13) to[short, -*] (13.75,13); 
\draw (11.25,10.5) to[short, -*] (11.25,10.5); 
\draw (13.75,10.5) to[short, -*] (13.75,10.5); 
\draw [line width=1pt, short] (11.25,5.5) .. controls (11.25,6.5) and (11.25,6.5) .. (11.25,7.5); 
\draw [line width=1pt, short] (11.25,5.5) to (13.75,5.5); 
\draw [line width=1pt, short] (13.75,5.5) .. controls (13.75,6.5) and (13.75,6.5) .. (13.75,7.5); 
\draw [line width=1pt, short] (11.25,7.5) to (12.5,8.75); 
\draw [line width=1pt, short] (12.5,8.75) to (13.75,7.5); 
\draw (11.25,7.5) to[short, -*] (11.25,7.5); 
\draw (12.5,8.75) to[short, -*] (12.5,8.75); 
\draw (13.75,7.5) to[short, -*] (13.75,7.5); 
\draw (13.75,5.5) to[short, -*] (13.75,5.5); 
\draw (11.25,5.5) to[short, -*] (11.25,5.5); 
\draw [line width=0.8pt, short] (16.25,12.25) to (16.25,10.5); 
\draw [line width=0.8pt, short] (16.25,10.5) to (18,10.5); 
\draw [line width=0.8pt, short] (18,10.5) to (18,12.25); 
\draw [line width=0.8pt, short] (16.25,12.25) to (18,12.25); 
\draw [line width=0.8pt, short] (16.25,12.25) to (17.5,13); 
\draw [line width=0.8pt, short] (18,12.25) to (19.25,13); 
\draw [line width=0.8pt, short] (18,10.5) to (19.25,11.25); 
\draw [line width=0.8pt, short] (19.25,13) to (19.25,11.25); 
\draw [line width=0.8pt, short] (17.5,13) to (19.25,13); 
\draw [line width=0.8pt, dashed] (17.5,13) to (17.5,11.25); 
\draw [line width=0.8pt, dashed] (16.25,10.5) to (17.5,11.25); 
\draw [line width=0.8pt, dashed] (17.5,11.25) to (19.25,11.25); 
\draw [line width=0.8pt, short] (16.25,11.75) to (18,11.75); 
\draw [line width=0.8pt, short] (18,11) to (19.25,11.75); 
\draw (16.25,10.5) to[short, -*] (16.25,10.5); 
\draw (16.25,12.25) to[short, -*] (16.25,12.25); 
\draw (16.25,11.75) to[short, -*] (16.25,11.75); 
\draw (18,11.75) to[short, -*] (18,11.75); 
\draw (18,12.25) to[short, -*] (18,12.25); 
\draw (18,10.5) to[short, -*] (18,10.5); 
\draw (17.5,11.25) to[short, -*] (17.5,11.25); 
\draw (19.25,11.25) to[short, -*] (19.25,11.25); 
\draw (17.5,13) to[short, -*] (17.5,13); 
\draw (19.25,13) to[short, -*] (19.25,13); 
\draw (18,11) to[short, -*] (18,11); 
\draw (19.25,11.75) to[short, -*] (19.25,11.75); 
\draw [line width=0.8pt, short] (16.25,7.25) to (16.5,6.75); 
\draw [line width=0.8pt, short] (16.5,6.75) to (16.25,5.5); 
\draw [line width=0.8pt, short] (16.25,5.5) to (18,5.5); 
\draw [line width=0.8pt, short] (17,6.75) to (18,5.5); 
\draw [line width=0.8pt, short] (17,6.75) to (18,7.25); 
\draw [line width=0.8pt, short] (16.25,7.25) to (18,7.25); 
\draw [line width=0.8pt, short] (16.25,7.25) to (16.25,5.5); 
\draw [line width=0.8pt, short] (16.5,6.75) to (17,6.75); 
\draw [line width=0.8pt, short] (18,5.5) to (18.45,6.15);
\draw [line width=0.8pt, short] (18,7.25) to (18.45,6.15);
\draw [line width=0.8pt, short] (18,5.5) to (19.25,6.25);
\draw [line width=0.8pt, short] (19.25,6.25) to (19.05,6.48);
\draw [line width=0.8pt, short] (19.25,6.25) to (19.25,8);
\draw [line width=0.8pt, short] (18.45,6.15) to (19.05,6.48); 
\draw [line width=0.8pt, short] (17,6.75) to (18.45,6.15);
\draw [line width=0.8pt, short] (18,7.25) to (19.25,8);
\draw [line width=0.8pt, short] (19.25,8) to (19.05,6.48);
\draw [line width=0.8pt, short] (16.25,7.25) to (17.5,8);
\draw [line width=0.8pt, short] (17.5,8) to (19.25,8);
\draw [line width=0.8pt, dashed] (17.5,8) to (17.5,6.25);
\draw [line width=0.8pt, dashed] (16.25,5.5) to (17.5,6.25);
\draw [line width=0.8pt, dashed] (17.5,6.25) to (19.25,6.25);
\draw (16.25,7.25) to[short, -*] (16.25,7.25); 
\draw (16.5,6.75) to[short, -*] (16.5,6.75); 
\draw (16.25,5.5) to[short, -*] (16.25,5.5); 
\draw (17,6.75) to[short, -*] (17,6.75); 
\draw (18,5.5) to[short, -*] (18,5.5); 
\draw (18,7.25) to[short, -*] (18,7.25); 
\draw (17.5,6.25) to[short, -*] (17.5,6.25); 
\draw (19.25,6.25) to[short, -*] (19.25,6.25); 
\draw (17.5,8) to[short, -*] (17.5,8); 
\draw (19.05,6.48) to[short, -*] (19.05,6.48); %
\draw (19.25,8) to[short, -*] (19.25,8); %
\draw (18.45,6.15) to[short, -*] (18.45,6.15); 
\end{circuitikz}
}%
\caption{Freehedra and Goodwillie polytopes in low dimensions.}
\end{figure}
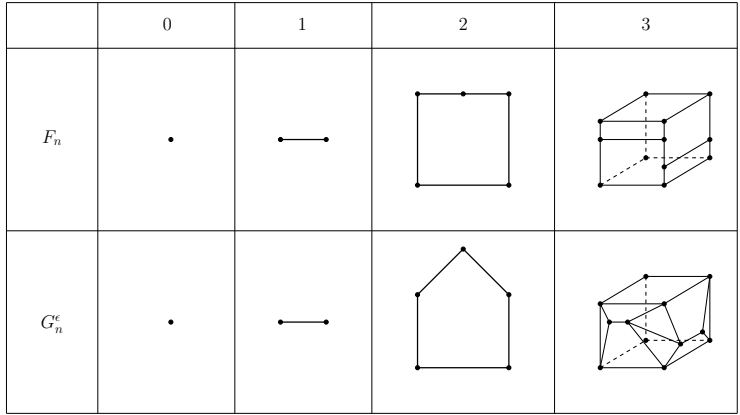
The Goodwillie polytopes decompose as a union of prisms glued along their faces, as the next result shows. 
\begin{proposition}\label{decomp}
    For any $\epsilon \in [0,\frac{1}{n-1})$, the polytope $G_n^{\epsilon}$ may be decomposed as a union
    $$G_n^{\epsilon} = \bigcup\limits_{S \subseteq \{2,\dots,n\}}G_n^{\epsilon,S},$$
    where each $G_{n}^{\epsilon, S}$ is a polytope combinatorially equivalent to $\Delta^{\abs{S}+1} \times I^{n-\abs{S}-1}$ and, for any two $S,T \subseteq \{2,\ldots,n\}$, $G_{n}^{\epsilon, S} \cap G_{n}^{\epsilon, T}$ is a common face of $G_{n}^{\epsilon, S}$ and $G_{n}^{\epsilon, T}$. If $\epsilon>0$, each facet of $G_n^{\epsilon}$ is a facet of some $G^{\epsilon,S}_n$ for exactly one $S \subseteq \{2,\ldots,n\}.$
\end{proposition}

\begin{proof}
For each $S\subseteq\set{2,\dots, n}$ define
    \[G_n^{0,S} := \{  (a_1,\dots,a_n) \in G_n^{0} \,\vert \, a_i \leq 0 \text{ if } i \in {S}, a_i \geq 0 \text{ if } i \notin {S}  \},\]
   so that
    \[G_n^{0} = \bigcup\limits_{{S}\subseteq \{2,\dots,n\} } G_n^{0,S}\, .\] Note that each $G^{0,S}_n$ is the convex hull of the set $V_{G^{0}_n}  \cap G^{0,S}_n$. Let $p_{\epsilon} \colon G_n^{\epsilon} \rightarrow G_n^0$ be the linear map determined by the map on vertices $V_{G^{\epsilon}_n} \to V_{G^{0}_n}$ sending $(a_1, \ldots, a_n)$ to $(a_1, \overline{a_2}, \ldots, \overline{a_n})$ where $\overline{a_j}= 0$ or $1$ if $a_j= \epsilon$ or $1-\epsilon$, respectively, and  $\overline{a_j}=a_j$ otherwise.
    Define $G^{\epsilon,S}_n$ as the convex hull of the points in $V_{G^{\epsilon}_n}$ that map into $G^{0,S}_n$ under $p_\epsilon$, so that
    \[
    G^{\epsilon,S}_n = \operatorname{conv}\{ v \in  V_{G^{\epsilon}_n} \, \vert \, p_\epsilon(v) \in G^{0,S}_n \} \textnormal{, and } G_n^{\epsilon} = \bigcup\limits_{{S}\subseteq \{2,\dots,n\} } G_n^{\epsilon,S}\, .
    \]
    The combinatorial type of $G^{\epsilon,S}_n$ does not change as $\epsilon$ varies in $[0,\frac{1}{n-1})$ (even though $G^0_n$ and $G^{\epsilon}_n$ are combinatorially different for $\epsilon >0$ and $n>2$).
   Suppose $S=\{i_1, \ldots, i_k\} \subseteq \{2, \ldots, n\}$ and $i_1< \cdots <i_k$, and let $\pi_S \colon G_n^{0, S} \to\mathbb{R}^{k+1}$ be the projection map $\pi_S(a_1,\ldots,a_n)= (a_1, a_{i_1}, \ldots a_{i_k})$. The set $\pi_S(G_n^{0, S})$ is the convex hull of $k+2$ affinely independent points, so it defines a polytope combinatorially equivalent to $\Delta^{k+1}$. The remaining coordinates, corresponding to indices not in $S$, parameterize a cube of dimension $n-k-1$. Thus $G_n^{0,S}$ (and consequently $G_n^{\epsilon,S}$ for any $\epsilon \in (0,\frac{1}{n-1})$) is combinatorially equivalent to $\Delta^{\abs{S}+1} \times I^{n-\abs{S}-1}$. See Figure \ref{decompfigure} for a picture when $n=3$. From the definition of $G^{\epsilon, S}$, it clearly follows that $G_{n}^{\epsilon, S} \cap G_{n}^{\epsilon, T}$ is a face of both $G_{n}^{\epsilon, S}$ and $G_{n}^{\epsilon, T}$.

The second statement follows from the fact that, if $\epsilon >0$, none of the hyperplanes determined by the facets of $I^n$ intersect the set $U_{G^{\epsilon}_n}$.
\end{proof}
\begin{figure}[h!] \label{decompfigure}
    \centering
    \includegraphics[width= \textwidth]{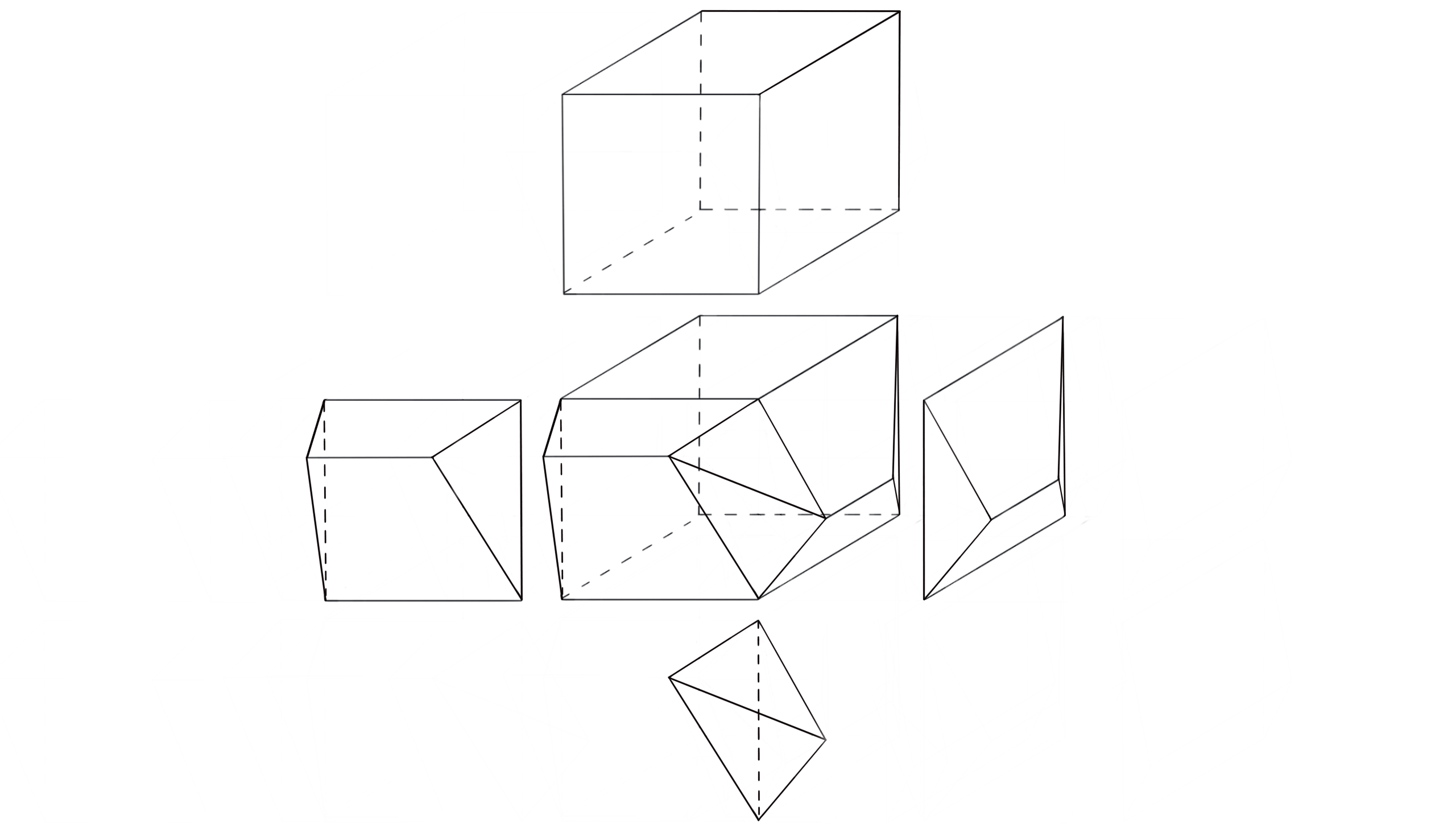}
    \caption{Example of the decomposition of $G_3^{0<\epsilon< \frac{1}{2}}$ into prisms.}
    \label{fig:G3_subdivision}
\end{figure}

\begin{corollary}
The polytope $\fancy{G}_n$ has the same number of facets as vertices. 
\end{corollary}
\begin{proof}
    Let $\#P$ denote the number of facets of a polytope $P$. When $0< \epsilon < \frac{1}{n-1}$, by the decomposition in Proposition \ref{decomp}, and the fact that every $G_n^{\epsilon,S}$ shares a facet with precisely $n-1$ other terms in the decomposition, we get
    \[
    \begin{array}{c}
       \# G_n^\epsilon = \sum\limits_{S\subseteq \{2, \dots , n\} } (\# G_n^S - (n-1))= \sum\limits_{S\subseteq \{2, \dots , n\} } (\#(\Delta^{|S|+1}\times I^{n-|S|-1}) - (n-1))  \\
    = \sum\limits_{S\subseteq \{2, \dots , n\} } (|S| + 2{n-|S|-1}) - (n-1))= \sum\limits_{k=0 }^{n-1} \binom{n-1}{k} (n-k-1) = (n+3)2^{n-2}.
    \end{array}
    \]
\end{proof}

Goodwillie polytopes ``project onto'' freehedra in a way that is compatible with their combinatorial structure. We now describe this relationship precisely. For any $\epsilon \in [0,\frac{1}{n-1})$, let $p_{\epsilon} \colon G_n^{\epsilon} \rightarrow G_n^0$ be the map determined by the map on vertices $V_{G^{\epsilon}_n} \to V_{G^{0}_n}$ sending $(a_1, \ldots, a_n)$ to $(a_1, \overline{a_2}, \ldots, \overline{a_n})$ where $\overline{a_j}= 0$ or $1$ if $a_j= \epsilon$ or $1-\epsilon$, respectively, and  $\overline{a_j}=a_j$ otherwise. In particular, $p_0=\text{id}$. Let $p \colon G_n^0 \rightarrow F_n$ be the continuous map that projects all points of $G_n^0$ with negative coordinates to $0$ on the negative coordinates and is the identity otherwise. The composition $\rho_{\epsilon} = p \circ p_{\epsilon} \colon G_n^{\epsilon} \rightarrow F_n$ exhibits $F_n$ as a deformation retract of $G^{\epsilon}_n$ and induces a map of abstract polytopes $\rho \colon \fancy{G}_n \rightarrow \fancy{F}_n$.\label{def:rho}

Computer calculations communicated by F. Chapoton suggest that each $\fancy{G}_n$ is a self-dual polytope, so its face vector is palindromic. For instance, the face vector of $\fancy{G}_4$ and $\fancy{G}_5$ are $(1, 28, 73, 73, 28, 1)$ and $(1, 64, 215, 304, 215, 64, 1)$, respectively. This conjecture and other combinatorial properties of the Goodwillie polytopes will be addressed elsewhere. 

\section{Categorical coalgebras and dg categories} \label{categoricalcoalgebrasanddgcategories}
\subsection{Categorical coalgebras} We describe a ``many-object'' version of a coaugmented dg coalgebra over an arbitrary ring $R$. When $R$ is a field, this notion fits into the ``categorical Koszul duality" framework developed in \cite{holstein-lazarev}. 

\begin{definition} \label{categoricalcoalgebra}
An \textit{$R$-categorical coalgebra} consists of a tuple $(C, \partial, \Delta,  h)$ where
\begin{enumerate}
    \item $(C,\Delta)$ is a graded coassociative counital coalgebra that is flat as an $R$-module
    \item $\partial:C \rightarrow C$ is a degree $-1$ coderivation of the coproduct $\Delta$
    \item $h: C \rightarrow R$ is a linear map of degree $-2$ satisfying $h\circ \partial =0$ and \[\partial^2 = (h \otimes \id) \circ (\Delta - \Delta^{\text{op}}),\] i.e. \textit{$h$ is a curvature for $(C,\partial,
    \Delta)$}
    \item The set \[\mathcal{S}(C) = \{ c \in C : \Delta(x)=x\otimes x \text{ and } \varepsilon(x)=1 \},\]
    of ``set-like'' elements in $C$ is non-empty and
    \[C_0 \cong R[\mathcal{S}(C)]\]
    \item The natural projection $\epsilon: C\rightarrow C_0$ satisfies $ \epsilon \circ \partial= 0 $
\end{enumerate}
\end{definition}
A categorical coalgebra $(C,\partial,\Delta,h)$ has a natural $C_0$-bi-comodule structure with structure maps given by
\[\begin{tikzcd}
	\rho_r:C & {C\otimes C} & {C\otimes C_0} \\
	\rho_l:C & {C\otimes C} & {C_0 \otimes C}
	\arrow["\Delta", from=1-1, to=1-2]
	\arrow["{\text{id} \otimes \epsilon}", from=1-2, to=1-3]
	\arrow["\Delta", from=2-1, to=2-2]
	\arrow["{\epsilon \otimes \text{id} }", from=2-2, to=2-3].
\end{tikzcd}\]

We denote by $C_0$-$\mathsf{biComod}$ the category of all $C_0$-bi-comodules whose underlying $R$-module is flat. This becomes a monoidal category when equipped with the cotensor product $\square_{C_0}$.

A morphism of categorical coalgebras $(C,\partial,\Delta,h)$ and $(C',\partial',\Delta', h')$ consists of a pair $f=(f_0,f_1)$ where
\begin{enumerate}
    \item $f_0:(C,\Delta)\rightarrow (C',\Delta')$ is a morphism of graded coalgebras 
    \item $f_1: C \rightarrow C_0^\prime$ is a $C_0^{\prime}$-bicomodule map of degree $-1$
\end{enumerate}
satisfying
\begin{align*}
    & f_0\circ \partial = \partial' \circ f_0 + (\overline{f}_1 \otimes f_0)\circ(\Delta - \Delta^{\text{op}}) \text{ and}\\
    & h'\circ f_0 = h+ \overline{f}_1 \circ \partial + (\overline{f}_1\otimes \overline{f}_1)\circ \Delta , 
\end{align*}
where $\overline{f}_1 = \varepsilon' \circ f_1$, and $\varepsilon'$ is the counit of $C'$. The composition of two morphisms $(f_0,f_1),(g_0,g_1)$ of categorical coalgebras is defined as
\begin{equation}
    (g_0,g_1)\circ(f_0,f_1)= (g_0 \circ f_0, g_1\circ f_0 + g_0 \circ f_1).
\end{equation}
We denote by $\mathsf{cCoalg}_R$ the category of categorical coalgebras.

\begin{example}\label{chainsondeltan} \textit{The categorical coalgebra of simplicial normalized chains on the $n$-simplex.} Let $\mathbb{\Delta}^n$ be the standard $n$-simplex considered as a simplicial set. Denote by $(C_\bullet(\mathbb{\Delta}^n), \partial, \Delta)$ be the dg $R$-coalgebra of normalized simplicial chains on $\mathbb{\Delta}^n$ with 
\[ \partial= \partial_k \colon C_k(\mathbb{\Delta}^n) \to  C_{k-1}(\mathbb{\Delta}^n) \]
the usual simplicial boundary map given by the alternating sum of the face maps
\[ \partial_k = \sum_{i=0}^k (-1)^i\partial_{k,i},\]
and 
\[ \Delta \colon C_\bullet(\mathbb{\Delta}^n) \to C_\bullet(\mathbb{\Delta}^n) \otimes C_\bullet(\mathbb{\Delta}^n) \]
the Alexander-Whitney diagonal approximation. For any $0 \leq k \leq n$, the $R$-module $C_k(\mathbb{\Delta}^n)$ has a canonical basis given by the non-degenerate simplices of $\mathbb{\Delta}^n_k$ all of which have non-degenerate faces. For any non-degenerate simplex $\sigma \in \mathbb{\Delta}^n_k$ define a new map
\[ \widetilde{\partial_k}(\sigma)= \sum_{i=1}^k \partial_{k,i}(\sigma) \]
by removing the first and last terms of $\partial_k$. This defines a new differential
\[ \widetilde{\partial} \colon C_\bullet(\mathbb{\Delta}^n) \to  C_{\bullet-1}(\mathbb{\Delta}^n),\]
which now satisfies $\epsilon \circ \widetilde{\partial}=0$. In fact, $(C_\bullet(\mathbb{\Delta}^n), \widetilde{\partial}, \Delta, h=0)$ defines a categorical coalgebra. See \cite[Example 4.8]{holstein-lazarev}.  
\end{example}

\begin{example}\label{categoricalchains} \textit{The categorical coalgebra of singular chains on a space.} The normalized singular chains on a topological space may be regarded as a categorical coalgebra as we now explain. Recall that the normalized singular chains functor 
\[C_{\bullet}: \mathsf{Top} \rightarrow \mathsf{dgCoalg}_R\]
assigns to a space $X$ its dg coalgebra of normalized singular chains $(C_{\bullet}(X),\partial , \Delta)$, where the coproduct is given by the Alexander-Whitney diagonal approximation. This does not define a categorical coalgebra with curvature $0$ since, in general, $\epsilon \circ \partial \neq 0$. The modification in the example above of removing the first and last term of the differential does not induce a well defined map since, in general, a non-degenerate simplex may have degenerate faces.  However, one may proceed as follows. Let \[e: R\langle \mathsf{Top}(\Delta^1,X) \rangle \rightarrow R\] be the linear map determined by sending degenerate $1$-simplices to $0 \in R$, and sending non-degenerate $1$-simplices to $1 \in R$. This induces a linear map on normalized chains $\widetilde{e}:C_1(X)\rightarrow R$. Define a linear map $\widetilde{\partial} \colon C_\bullet(X) \to C_{\bullet-1}(X)$ and a cochain $h: C_2(X) \rightarrow R$ by
\[
\begin{array}{cc}
    \widetilde{\partial} = \partial - (\id \otimes e - e\otimes \id) \circ \Delta    \\
      h = (e \otimes e) \circ \Delta + e \circ \partial
\end{array}
\]
A straightforward check yields that $\mathcal{C}(X)=(C_\bullet(X),\widetilde{\partial},\Delta,h)$ defines a categorical coalgebra. Furthermore, this construction defines a functor 
\[ \mathcal{C} \colon \mathsf{Top} \to \mathsf{cCoalg}_R.\]
\end{example}

\subsection{Differential graded categories} A \textit{differential graded (dg) category} is a category enriched over the monoidal category $(\mathsf{Ch}_R, \otimes_R)$ of dg $R$-modules. A \textit{non-negatively graded dg category} is a category enriched over the monoidal category $(\mathsf{Ch}^{\geq 0}_R, \otimes_R)$ of differential non-negatively graded $R$-modules. A morphism $f \colon A \to A'$ of dg categories is called a \textit{quasi-equivalence} if it induces quasi-isomorphisms on all dg $R$-modules of morphisms and if the induced functor of categories $H_0(f) \colon H_0(A) \to H_0(A')$, obtained by applying the $0$-th homology functor on morphisms, is an equivalence of categories.  We denote by $\mathsf{dgCat}_R$ ($\mathsf{dgCat}^{\geq 0}_R$) the category of small (non-negatively graded) dg categories over $R$. 

\begin{example} \label{categoriesofpaths}\textit{The dg category of paths on a space.} Any topological space $X$ gives rise to a dg category $\mathbf{P}^{\mathbb{I}}X$ as follows. Consider the category $\mathbf{P}X$ of (Moore) paths on a space $X$, whose objects are the points of $X$  and whose spaces of morphisms are defined as
\[\mathbf{P}X(x_0,x_1) = \{ (\gamma, l) \in PX : \gamma (0)=x_0, \gamma(l)=x_1 \}.\] Composition is given by concatenation of paths and adding corresponding parameters, and the identity at $x \in X$ is given by $(c_x, 0)$ where $ c_x \colon \{0\} \to X$ is the constant loop at $x$. This may be regarded as a topological category, i.e. a category enriched over the monoidal category $(\mathsf{Top}, \times)$,  by equipping each morphism set with the topology induced by the compact-open topology. Applying the monoidal functor of normalized singular cubical chains $C_{\bullet}^{\mathbb{I}} \colon \mathsf{Top} \to \mathsf{Ch}_R$ on each morphism space of $\mathbf{P}X$ we obtain a dg category over $R$, which we will call the \textit{dg category of paths on} $X$ and denote by $\mathbf{P}^{\mathbb{I}}X$. This construction yields a functor \[\mathbf{P}^{\mathbb{I}}: \mathsf{Top}\rightarrow \mathsf{dgCat}_R^{\geq 0}\] that sends weak homotopy equivalences of spaces to quasi-equivalences of dg categories.
\end{example}


\subsection{Cobar functor}\label{cobarfunctor} The categorical coalgebra of chains and the dg category of paths are related via the \textit{cobar functor}
\[ \mathbf{\Omega}: \mathsf{cCoalg}_R \rightarrow \mathsf{dgCat}_R^{\geq 0},\]
which we now recall. 

Let $(C,\Delta , \partial , h)$ be a categorical coalgebra.  For any $x \in \mathcal{S}(C)$ denote by $i_x: R \to C_0=R[\mathcal{S}(C)]$ be the linear map determined by $i_x(1_R)= x$. This gives rise to a $C_0$-bicomodule structure on $R$ through the  maps $$R \cong R \otimes R \xrightarrow{i_x \otimes \text{id}_{R}} C_0 \otimes R$$ and $$R \cong R \otimes R \xrightarrow{\text{id}_R \otimes i_x} R \otimes C_0.$$ We denote this $C_0$-bicomodule by $R_x$ and its generator by $\text{id}_x.$

Let $\overline{C}= \bigoplus\limits_{i\geq1} C_i$, so that $C= \overline{C} \oplus C_0$, and let $s^{-1}\overline{C}$ the graded $R$-module obtained by shifting $\overline{C}$ by $-1$. We have the following three degree $-1$ maps
\begin{enumerate}
\item $\overline{\partial}\colon s^{-1}\overline{C} \to s^{-1}\overline{C}$,
\item $\overline{\Delta}\colon s^{-1}\overline{C} \to s^{-1}\overline{C} \otimes s^{-1}\overline{C}$, and
\item $\overline{h}\colon s^{-1}\overline{C} \xrightarrow{s^{+1}} \overline{C} \xrightarrow{\rho_r} C \otimes C_0 \xrightarrow{h \otimes \textit{id}} R \otimes C_0 \cong C_0$
\end{enumerate}
induced by $\partial, \Delta$ and $h$, respectively. 

The objects of $\mathbf{\Omega}C$ are defined to be all the elements of $\mathcal{S}(C)$. For any two $x,y \in \mathcal{S}(C)$ define a non-negatively graded $R$-module by
$$\mathbf{\Omega}C(x,y)= \bigoplus_{i=0}^{\infty} R_x \underset{C_0}{\msquare} (s^{-1}\overline{C})^{\square i} \underset{C_0}{\msquare} R_y,$$
where $(s^{-1}\overline{C})^{\square i} $ denotes the $i$-fold cotensor product of $C_0$-bicomodules and $(s^{-1}\overline{C})^{\square 0}=C_0$.
We will use the notation $\{c_1| \cdots |c_p\}$
to denote a generator
\[ \text{id}_x \square s^{-1}c_1 \square \cdots \square s^{-1}c_p \square \text{id}_y \in \mathbf{\Omega}C(x,y). \] 
%
The differential $$D_{x,y}: \mathbf{\Omega}C(x,y)_k \to \mathbf{\Omega}C(x,y)_{k-1}$$ is defined by extending
\[\overline{h} + \overline{\partial} + \overline{\Delta} \colon R_x \square s^{-1}\overline{C} \square R_y \to
(R_x \square C_0 \square R_y )\oplus
(R_x \square s^{-1}\overline{C} \square R_y) \oplus (R_x \square (s^{-1}\overline{C})^{\square 2} \square R_y)
\]
as a ``derivation" to monomials of arbitrary length. Note that when we apply these maps the Koszul sign rule introduces signs. It follows directly from the definition of a categorical coalgebra that $D_{x,y} \circ D_{x,y}=0$. 
The composition in $\mathbf{\Omega}C$ is given by concatenation of monomials. For every $x \in \mathcal{S}(C)$, $1_{R} \in  R_x \cong R_x \square C_0 \square R_x \subset \mathbf{\Omega}C(x,x)_0$ is the identity morphism.

Given a morphism $f= (f_0,f_1): C \to C'$ between categorical coalgebras, define a morphism
\[\mathbf{\Omega}f= \mathbf{\Omega}(f_0,f_1)\colon \mathbf{\Omega}C \to \mathbf{\Omega}(C') \]
of dg categories as follows. Since $f_0: C\to C'$ is a map of coalgebras, $f_0$ restricts to a map of sets $\mathcal{S}(C) \to \mathcal{S}(C')$, which defines the functor $\mathbf{\Omega}(f_0,f_1)$ on objects. For any two $x,y \in \mathcal{S}(C) $
define 
\[ \mathbf{\Omega}(f_0,f_1)_{x,y} \colon \mathbf{\Omega}C(x,y) \to \mathbf{\Omega}(C')(f_0(x), f_0(y)) \]
by extending the map 
\begin{eqnarray*} R_x  \square  s^{-1} \overline{C} \square R_y \longrightarrow (R_{f_0(x)} \square s^{-1} \overline{C} \square R_{f_0(y)}) \oplus (R_{f_0(x)} \square C_0' \square  R_{f_0(y)})
\\
\{c\} \longmapsto \{f_0(c)\} + \text{id}_{f_0(x)} \square f_1(c) \square \text{id}_{f_0(y)}
\end{eqnarray*}
``multiplicatively" to monomials $\{c_1 | \cdots | c_p \}$ of arbitrary length. 
Note that $R_{f_0(x)} \square C' _0 \square  R_{f_0(y)}$ is a non-trivial $R$-module if and only if $f_0(x)=f_0(y)$, in which case it is isomorphic to $R$. Hence, $\text{id}_{f_0(x)} \square f_1(c) \square \text{id}_{f_0(y)}$ may be identified with a scalar. It follows directly from the definition of morphisms between categorical coalgebras that each $\mathbf{\Omega}(f_0,f_1)_{x,y}$ is a chain map and that composition is compatible with differentials. We shall consider categorical coalgebras under the following notion of weak equivalence.
\begin{definition}
A morphism of categorical coalgebras $f \colon C \to C'$ is called a $\mathbf{\Omega}$-\textit{quasi-equivalence} if $\mathbf{\Omega}f \colon \mathbf{\Omega}C \to \mathbf{\Omega}C'$ is a quasi-equivalence of dg categories.
\end{definition}

\subsection{The Hochschild complex of a dg category} Let $A$ be a dg $R$-category all of whose morphisms complexes $A(x,y)$ are $R$-flat (i.e. $A$ is locally $R$-flat). Denote the coalgebra $R[\text{Obj}(A)]$ by $A_0$. We regard $A$ as a monoid $\mathcal{M}({A})$ in the monoidal category of $(A_0$-$\mathsf{biComod}, \square_{A_0})$ as follows. The underlying dg $R$-module is given by
    \be
        \mathcal{M}({A}) = \bigoplus_{x,y \in \text{Obj}(A)}{A}(x,y).
    \ee
 The source and target maps of $A$ induce an $A_0$-bicomodule structure 
\[ \mathcal{M}(A) \to A_0 \otimes \mathcal{M}(A) \] and
\[\mathcal{M}(A) \to \mathcal{M}(A) \otimes A_0 .
\]
The monoid structure 
\[\mathcal{M}(A) \underset{A_0}{\msquare} \mathcal{M}(A) \to \mathcal{M}(A) \]
is induced by the composition of morphisms in $A$ and the unit map
$A_0\to \mathcal{M}(A)$ is determined by $x \mapsto \text{id}_x \in A(x,x)_0$ for all $x \in \text{Obj}(A).$

We will also make use of the following notion for subsequent constructions. Let $C$ be a dg $R$-coalgebra, $A$ be a monoid in the monoidal category $(C\text{-}\mathsf{biComod}, \square_C)$ and $E,F$ be dg right and left $C$-comodules, equipped with right and left dg $A$-module structures, respectively. The \textit{tensor product of $E$ and $F$ over $A$} is defined to be the dg $R$-module
    \be\label{def:tensor-over-monoid}
        E\underset{A}{\bigotimes} F = \text{coker} (\rho_E \underset{C}{\square} \id_F - \id_E \underset{C}{\square} \rho_F : E\underset{C}{\square} A \underset{C}{\square} F \rightarrow E\underset{C}{\square} F),
    \ee
    where $\rho_E \colon E\underset{C}{\square} A \to E $ and $\rho_F \colon A\underset{C}{\square} F \to F$ are the right and left $A$-module structure maps of $E$ and $F$, respectively.

\begin{definition}
    Let $A$ be a dg category and let $(M,d_M)$ and $(N,d_N)$ be right and left dg modules, respectively, over $\mathcal{M}(A)$ in the monoidal category $(A_0\text{-}\mathsf{biComod}, \square_{A_0})$. Define the \textit{two-sided bar construction of $M$ and $N$ over $A$} as the dg $A_0$-bicomodule 
    \[\mathcal{B}(M,\mathcal{M}(A),N)\]
    whose underlying graded module is 
    \be
        \bigoplus\limits_{i=0}^{\infty} \Big( M \underset{A_0}{\square} \big(s^{+1}\overline{\mathcal{M}(A)}\big)^{\square i} \underset{A_0}{\square} N \Big),
    \ee
    where $\overline{\mathcal{M}(A)} = \mathcal{M}(A)/u(A_0)$ and $u:A_0\rightarrow \mathcal{M}(A)$ the unit map. We denote generators
    \[m\square s^{+1}a_1 \square \cdots \square s^{+1}a_p\square n,\]
    where $m\in M, n\in N$, and $a_i \in \overline{\mathcal{M}(A)}$ for each $i$, by the usual ``bar'' notation
    \[ m[a_1 \vert \cdots \vert a_p] n .\]
    The grading is defined by 
    \[ m[a_1 \vert \cdots \vert a_p] n \in \mathcal{B}(M,\mathcal{M}(A),N)_r \text{ if } r=\abs{n}+\abs{m}+ \sum\limits_{i=1}^p \abs{a_i} + p ,\]
     where $\abs{n},\abs{m},\abs{a_i}$, denotes the degree in $N,M$, and $A$, accordingly.
     The differential
     \be
        \delta_{M,\mathcal{M}(A),N} : \mathcal{B}(M,\mathcal{M}(A),N)_\bullet \rightarrow \mathcal{B}(M,\mathcal{M}(A),N)_{\bullet -1}
     \ee
     is defined by
     \be
     \delta_{\mathsf{M},\mathcal{M}(A),\mathsf{N}}= d_{\mathsf{M}} \square \text{id}_{\mathcal{M}(A)} \square \text{id}_{\mathsf{N}} + \text{id}_{\mathsf{M}} \square d_{\mathcal{M}(A)}\square \text{id}_{\mathsf{N}} + \text{id}_{\mathsf{M}}  \square \text{id}_{\mathcal{M}(A)} \square d_{\mathsf{N}} + \theta, \ee
    where $d_{\mathsf{M}}, d_{\mathsf{N}}$, and $d_{\mathcal{M}(A)}$ are the differentials of $\mathsf{M}, \mathsf{N}$, and $\mathcal{M}(A)$, respectively, and $\theta$ is given by the following formula
\begin{align*}
    \theta( m [a_1 \big| \cdots \big| a_p]n)= ~ & m\cdot a_1 [a_2\big| \cdots \big| a_p] n \\&+ \sum_{i=1}^{p-1}  
    \pm m [a_1\big| \cdots \big| a_i \cdot a_{i+1} \big| \cdots \big| a_p]n \\ &\pm m [a_1\big| \cdots \big| a_{p-1}]a_p\cdot  n,
\end{align*}
where, as usual, the signs are given by the Koszul sign rule.
A routine computation yields $\delta_{\mathsf{M},\mathcal{M}(A),\mathsf{N}}^2=0$. 
\end{definition}
\begin{definition}
    Let $A$ be a dg category. The \textit{Hochschild complex of} $A$ is the dg $R$-module
    \be
        (\textnormal{Hoch}_\bullet(A), \delta) = \mathcal{B}(\mathcal{M}(A), \mathcal{M}(A), \mathcal{M}(A) ) \underset{\mathcal{M}(A) \otimes \mathcal{M}(A)^{\text{op}}}{\bigotimes} \mathcal{M}(A).
    \ee
  We unravel this definition in more detail. As a graded $R$-module, this is the same as
    \be
    \bigoplus\limits_{i=0}^{\infty} \Big( \big(s^{+1}\overline{\mathcal{M}}(A)\big)^{\square i} \underset{A_0 \otimes A_0^{\text{op}}}{\msquare}  \mathcal{M}(A)  \Big), 
    \ee
 so generators are given by monomials written as
    \[ [a_1 \vert \cdots \vert a_p]a_{p+1}, \]
    where $a_{p+1} \in \mathcal{M}(A)$, $a_i \in \overline{\mathcal{M}}(A)$ for $i=1, \dots, p$, and the source and target of consecutive $a_i$'s coincide, together with the source of the zeroth and target of the last terms, i.e. if $\mathbf{s}$ and $\mathbf{t}$ denote source and target, respectively, then
    \[ \mathbf{s}(a_i) = \mathbf{t}(a_{i-1}) \text{ for } i= 2,\dots, p, \text{ and } \mathbf{s}(a_{1})=\mathbf{t}(a_{p+1}).\]
    The total grading is given by
    \be [a_1 \vert \cdots \vert a_p ]a_{p+1} \in \text{Hoch}_r(A) \text{ if } r = \big( \sum\limits_{i=1}^{p+1} \abs{a_i} \big) + p ,
    \ee
    where $\abs{a_i}$ is the degree of $a_i$ in $\mathcal{M}(A)$. The differential
    \[ \delta: \Hoch_\bullet (A) \rightarrow \Hoch_{\bullet -1}(A) \]
    is defined on generators $ [a_1 \vert \cdots \vert a_p ]a_{p+1}$ by
    \begin{align*}
        \delta ( [a_1 \vert \cdots \vert a_p]a_{p+1}) = ~& \sum\limits_{i=1}^{p+1} \pm [a_1 \vert \cdots \vert d_{\mathcal{M}(A)} a_i \vert \cdots \vert a_p]a_{p+1} 
        \pm [a_2 \vert \cdots \vert a_p]a_{p+1}a_1 \\ &+ \sum\limits_{i=1}^{p-1} \pm [a_1 \vert \cdots \vert a_i a_{i+1} \vert \cdots \vert a_p]a_{p+1} \pm [a_1|\cdots a_{p-1}]a_pa_{p+1},
    \end{align*}
    where the signs are given by the Koszul sign rule and where we have written composition as concatenation. 
\end{definition}
\subsection{The coHochschild complex of a categorical coalgebra} \label{cohochschild} We now describe a construction that generalizes the coHochschild complex of a counital dg coalgebra, as discussed in \cite{Doi, hesscohochschild, hess-shipley}, to the context of categorical coalgebras. The construction is ``Koszul dual" to the Hochschild complex of a dg category, as stated \ref{hoch-cohoch}. 

\begin{definition}\label{def:q-complex}
    Let $(C, \partial_C, \Delta_C,h_C)$ be a categorical coalgebra and let $(M,d_M)$ and $(N,d_N)$ be right and left modules, respectively, over $\mathcal{M}(\mathbf{\Omega}C)$ in the category $(C_0\text{-}\mathsf{biComod}, \square_{C_0})$. (Note that, by definition of $\mathbf{\Omega}$, $(\mathbf{\Omega}C)_0=C_0$.)
    Define a graded $R$-module by
    \be
    \mathcal{Q}(\mathsf{M}, C, \mathsf{N}):= \mathsf{M} \underset{C_0}{\square} C \underset{C_0}{\square}  \mathsf{N} 
    \ee
    and a degree $-1$ linear map
    \be  d_{\mathcal{Q}}=
d_{\mathsf{M}} \square \text{id}_{C} \square \text{id}_{\mathsf{N}} + \text{id}_{\mathsf{M}} \square \partial_C \square \text{id}_{\mathsf{N}} + \text{id}_{\mathsf{M}}  \square \text{id}_{C} \square d_{\mathsf{N}} + \Theta,  \ee
where
\be \Theta(m \square c \square n)=
\pm ( m \cdot \{c'\} ) \square c'' \square n \pm m \square  c' \square (\{c''\} \cdot n).\ee
    and the symbols $c',c''$ are terms coming from the coproduct of $C$, expressed in Sweedler notation.
\end{definition}
\begin{definition}
   For any categorical coalgebra $C$ define the \textit{coHochschild complex of} $C$ by
    \be (\coHoch_\bullet(C),d) =  \mathcal{M}(\mathbf{\Omega}C) \underset{\mathcal{M}(\mathbf{\Omega}C) \otimes \mathcal{M}(\mathbf{\Omega}C)^{\text{op}}}{\bigotimes}  \mathcal{Q}(\mathcal{M}(\mathbf{\Omega}C), C, \mathcal{M}(\mathbf{\Omega}C)), \ee
    where $d$ is induced by $d_{\mathcal{M}(\mathbf{\Omega}C)}$ and $d_{\mathcal{Q}}$.
\end{definition}
We unravel this definition in more detail. The underlying graded $R$-module is given by
    \be C\underset{C_0\otimes C_0^{\text{op}}}{\square } \mathcal{M}(\mathbf{\Omega}C), \ee
    so it is generated by monomials $c_0 \square s^{-1}c_1 \square \cdots \square s^{-1}c_p$,
    which we write as \[c_0 \{ c_1 | \cdots | c_p \},\] where $c_0 \in C$, $c_i \in \overline{C}=C_{>0}$ for $i=1,\dots, p$, and $c_p \otimes c_0 \in C \underset{C_0}{\square} C$. The total grading is given by
    \be  c_0 \{ c_1 | \cdots | c_p \} \in \coHoch_r(C) \text{ if } r= \big(\sum\limits_{i=0}^{p} \abs{c_i}\big) -p ,  \ee
    where $\abs{c_i}$ is the degree of $c_i$ in $C$. The differential 
    \[ d \colon \textnormal{coHoch}_\bullet(C) \to \textnormal{coHoch}_{\bullet-1}(C)\]  is defined on generators by
    \begin{align*} d( c_0 \{ c_1 | \cdots | c_p \})=& ~\partial_Cc_0 \{ c_1 |\cdots | c_p \}
+\sum_{i=1}^p \pm c_0 \{ c_1 | \cdots|\partial_Cc_i | \cdots | c_p\} \\
&+\sum_{i=1}^{p} \pm c_0 \{c_1 | \cdots | \overline{h_C}(c_i) | \cdots |c_p\}
\\
&+ \sum_{(c_0)} \pm c_0' \{ c_0''| c_1 | \cdots | c_p\}
+ \sum_{i=1}^p \sum_{(c_i)} \pm c_0 \{ c_1 | \cdots |c_i' | c_i''|  \cdots  |c_p\}\\
&+ \sum_{(c_0)} \pm c_0'' \{ c_1 | \cdots| c_p |c_0'\}. 
\end{align*} 
where the signs are given by the Koszul sign rule, and we use Sweedler notation for the coproduct $\Delta_C$.
Even if the graded $R$-module $\mathcal{Q}(M, C, N)$ may not be a chain complex in general, when $M=N=\mathcal{M}(\mathbf{\Omega}C)$ a routine computation (using the curvature term) yields that $d_{\mathcal{Q}}$ does square to zero. Passing to $\coHoch_\bullet(C)$ we obtain an honest differential $d$, making $\coHoch_\bullet(C)$ into a dg $R$-module. This construction is functorial with respect to morphisms of categorical coalgebras.

\subsection{A natural chain contraction between the Hochschild and coHochschild complexes} \label{comparingresolutions}
We define a chain contraction between the coHochschild complex of a categorical coalgebra $C$ and the Hochschild complex of the dg category $\mathbf{\Omega}C$. This will be induced by a more fundamental chain contraction between \[\mathcal{Q}(\mathcal{M}(\mathbf{\Omega}C), C, \mathcal{M}(\mathbf{\Omega}C))\text{ and }\mathcal{B}(\mathcal{M}(\mathbf{\Omega}C) , \mathcal{M}(\mathbf{\Omega}C ), \mathcal{M}(\mathbf{\Omega}C) ).\] 

\begin{definition}\label{def:alpha-pi}
  For any categorical coalgebra $C$, we define maps 
\[
\begin{tikzcd}
	 {\mathcal{B}(\mathcal{M}(\mathbf{\Omega}C), \mathcal{M}(\mathbf{\Omega}C), \mathcal{M}(\mathbf{\Omega}C))}\arrow[loop left, distance=3em, start anchor={[yshift=-1ex]west}, end anchor={[yshift=1ex]west}]{}{H}& {\mathcal{Q}(\mathcal{M}(\mathbf{\Omega}C), C, \mathcal{M}(\mathbf{\Omega}C)).}
	\arrow["\pi", shift left=2, from=1-1, to=1-2]
	\arrow["\alpha", shift left=2, from=1-2, to=1-1]
\end{tikzcd}
\]
as follows.

The map $\pi$ is defined on generators $a_0[a_1 \big| \cdots \big| a_p]a_{p+1} $ by
\begin{align*}
    &{\pi( a_0 [\ ]  a_{1})= \operatorname{proj}_{C_0} \circ \rho_{\mathcal{M}(\mathbf{\Omega}C))}^{left}(a_0) = \operatorname{proj}_{C_0} \circ \rho_{\mathcal{M}(\mathbf{\Omega}C))}^{right}(a_1)\ ,}  \\
    &\pi( a_0 [a_1 \big| \cdots \big| a_p]  a_{p+1})= 0 \text{ if $p>1$ }, 
\end{align*}
and when $p=1$, writing $a_1=\{c_1| \cdots |c_q\}$, 
\begin{equation}
   \pi(a_0 [ \{c_1 | \cdots |c_q\}]a_2) = 
       \sum_{i=1}^q a_0 \{c_1| \cdots |c_{i-1}\} \square [\{c_i\}] \square \{c_{i+1} | \cdots | c_q\} a_2 \text{, if } q>0,
\end{equation}
The map $\alpha$ is defined on generators $a \square c \square b$ by
\begin{align*}
    \alpha(a \square c \square b)= &~a [\{ c\} ] b + \sum a [\{c'\} \big| \{c''\}]b \\
    &+ \sum a [\{c'\} \big| \{c''\} \big| \{c'''\}]b + \cdots,
\end{align*} 
where we have used Sweedler notation for the coproduct of $C$. Note that $\alpha $ is well defined since the iterated coproduct eventually vanishes by degree reasons.

Define \[H:\mathcal{B}(\mathcal{M}(\mathbf{\Omega}C), \mathcal{M}(\mathbf{\Omega}C), \mathcal{M}(\mathbf{\Omega}C))  \to \mathcal{B}(\mathcal{M}(\mathbf{\Omega}C), \mathcal{M}(\mathbf{\Omega}C), \mathcal{M}(\mathbf{\Omega}C)) \] to be the degree $+1$ linear map given on a generator $a_0[a_1 \big| \cdots \big| a_p]a_{p+1}$ as follows. Write $a_1=\{c_1|\cdots|c_m\}$ and let
\[ H( a_0[\{c_1|\cdots|c_m\} \big| a_2 \big| \cdots \big| a_p]a_{p+1}) =0  \text{ if $m<2$ } ,\]
\begin{eqnarray*}
H(a_0[ \{c_1|c_2\} \big| a_2 \big| \cdots \big| a_p]a_{p+1})= a_0[\{c_1\} \big| \{c_2\} \big| a_2 \big| \cdots \big| a_p]a_{p+1} + \\
\sum a_0[\{c_1'\} \big| \{c_1''\} \big| \{c_2\} \big| a_2 \big| \cdots \big| a_p]a_{p+1} + \\
\sum a_0[\{c_1'\} \big| \{c_1''\} \big|\{c_1'''\}\big| \{c_2\} \big| a_2 \big| \cdots \big| a_p]a_{p+1} + \cdots \text{ if $m=2$, }
\end{eqnarray*}
and, if $m >2$, let
\begin{eqnarray*}
H( a_0[\{c_1|\cdots|c_m\} \big| a_2 \big| \cdots \big| a_p]a_{p+1}) =
\\
  \sum_{i=1}^{m}a_0 \cdot \{c_1| \cdots |c_{i-1}\} [ \{c_i \} \big| \{c_{i+1}| \cdots | c_m\} \big| a_2 \big| \cdots \big| a_p]a_{p+1} +
\\
\sum_{i=1}^{m-1}a_0 \cdot \{c_1| \cdots |c_{i-1}\} [ \{c_i' \} \big| \{c_i''\} \big| \{c_{i+1}| \cdots | c_m\} \big| a_2 \big| \cdots \big| a_p]a_{p+1} +
\\
\sum_{i=1}^{m-1}a_0\cdot \{c_1| \cdots |c_{i-1}\} [ \{c_i' \} \big| \{c_i''\} \big| \{c_i'''\} \big| \{c_{i+1}| \cdots | c_m\} \big| a_2 \big| \cdots \big| a_p]a_{p+1} + \cdots
\end{eqnarray*}

The following proposition follows from a routine (but tedious) computation, which may be found in \cite{tolosa-thesis}.

\begin{proposition}\label{chaincontraction} The maps $\alpha$ and $\pi$ are natural chain maps and $\pi \circ \alpha=\textnormal{id}_{\mathcal{Q}}$. Furthermore, $H$ is a chain homotopy between $\alpha \circ \pi $ and $\textnormal{id}_{\mathcal{B}}$, i.e. the data $(\pi, \alpha, H)$ defines a natural chain contraction of chain complexes. 
\end{proposition}

\end{definition}
The maps $\pi$ and $\alpha$ induce chain maps 
\begin{eqnarray}\label{overlinepi}
\overline{\pi}= \pi \underset{\mathcal{M}(\mathbf{\Omega}C) \otimes \mathcal{M}(\mathbf{\Omega}C)^{\text{op}}}{\bigotimes} \text{id}_{\mathcal{M}(\mathbf{\Omega}C)} \quad\text{ and }\quad \overline{\alpha}= \alpha \underset{\mathcal{M}(\mathbf{\Omega}C) \otimes \mathcal{M}(\mathbf{\Omega}C)^{\text{op}}}{\bigotimes} \text{id}_{\mathcal{M}(\mathbf{\Omega}C)}
\end{eqnarray}

between Hochschild and coHochschild complexes. The chain homotopy $H$ also induces a degree $+1$ linear map \[\overline{H}= H  \underset{\mathcal{M}(\mathbf{\Omega}C) \otimes \mathcal{M}(\mathbf{\Omega}C)^{\text{op}}}{\bigotimes} \text{id}_{\mathcal{M}(\mathbf{\Omega}C)}.\]
The following is an immediate consequence of  Proposition \ref{chaincontraction}.

\begin{corollary}\label{hoch-cohoch} For any categorical coalgebra $C$, we have a natural chain contraction of chain complexes
\[
\begin{tikzcd}
	 {(\textnormal{Hoch}_\bullet( \mathbf{\Omega}(C), \delta) }\arrow[loop left, distance=3em, start anchor={[yshift=-1ex]west}, end anchor={[yshift=1ex]west}]{}{\overline{H}}& {(\textnormal{coHoch}_\bullet(C), d).}
	\arrow["\overline{\pi}", shift left=2, from=1-1, to=1-2]
	\arrow["\overline{\alpha}", shift left=2, from=1-2, to=1-1]
\end{tikzcd}
\]
\end{corollary}
\section{Chain models for path spaces} \label{chainmodels}
\subsection{Chains}
In this section, we introduce two models for the dg $R$-module of chains on a space. In the first model, chains are parameterized by the cartesian product $F_n \times I^m$ of a freehedron and a cube. In the second model, chains are parameterized by the cartesian product $\Delta^n \times I^m$ of a simplex and a cube. These arise as the natural setting to relate (co)Hochschild constructions and loop spaces.

Following the labelling of the faces of a freehedron described in \ref{faces_and_breaks}, we introduce the following notation for formal sums of the faces of $F_n$. For any $n$-string \[s = \{ s_0 | \cdots |s_{l-1} \} s_l \{ s_{l+1} | \cdots | s_k \}\] let
    \begin{align*}
        &d^{0}_i(s) = \sum \pm \{ s_0 | \cdots |s_{l-1} \} s_l \{ s_{l+1} | \cdots | s_{i}^{1} | s_{i}^{2}|\cdots| s_k \}, \quad 0\leq i\leq k, ~i\neq l,l+1 , \\
        &d^{0}_{l}(s) = \sum \pm \{ s_0 | \cdots |s_{l-1} | s_l^1\} s_l^2 \{ s_{l+1} | \cdots | s_k \},\\
        &d^{0}_{l+1}(s) = \sum \pm \{ s_0 | \cdots |s_{l-1}\} s_l^1\{ s_l^2 | s_{l+1} | \cdots | s_k \},\\
    \end{align*}
    where 
    \begin{itemize}
        \item  the first sum runs over all possible partitions of $s_{i}$ of the form $ s_{i}^{1} = \{ a\in s_i | a\leq x\}$, and $s_{i}^{2} = \{ a\in s_i | a\geq x\}, $ for some $x$ with $\min(s_{i}) < x< \max(s_{i})$,
        \item the second sum runs over all partitions of the form $ s_{l}^{1} = \{ a\in s_l | a\leq x\}$, and $s_{l}^{2} = \{ a\in s_l | a\geq x\}$, for some $x$ with  $\min(s_{l}) \leq x< \max(s_{l})$ , and
        \item the third sum runs over all partitions of the form $ s_{l}^{1} = \{ a\in s_l | a\leq x\}$, and $s_{l}^{2} = \{ a\in s_l | a\geq x\}$, for some $x$ with  $\min(s_{l}) < x \leq \max(s_{l})$.
        \end{itemize} 
        
        Similarly, we define the formal sum
    \begin{equation*}
        d^{1}_{i}(s) = \sum\limits_{ j \in s_i \setminus \{ \text{min}(s_i), \text{max}(s_i)\} }   \pm \{ s_0 |\cdots | \hat{s}_i^{j} |\cdots |s_{l-1} \} s_l \{ s_{l+1} | \cdots | s_k \}, \quad i=0,\dots , k,
    \end{equation*}
    where $\hat{s}_i^{j}$ denotes the subset of the ordered set $s_i$ obtained by omitting the $j$-th element. We now explain how the signs above are determined. The key observation is that any  subset $s_j$ of $\{0,\ldots,n\}$ corresponds exactly to a non-degenerate simplex in the simplicial set $\mathbb{\Delta}^n$, so a string $s = \{ s_0 | \cdots |s_{l-1} \} s_l \{ s_{l+1} | \cdots | s_k \}$ corresponds to a generator of \begin{eqnarray}\label{Qdelta} \mathcal{Q}( \mathcal{M}(\mathbf{\Omega}C_\bullet(\mathbb{\Delta}^n)), C_\bullet(\mathbb{\Delta}^n), \mathcal{M}(\mathbf{\Omega}C_\bullet(\mathbb{\Delta}^n))),
    \end{eqnarray} 
  where $C_\bullet(\mathbb{\Delta}^n)$ is the categorical coalgebra of example \ref{chainsondeltan}. In fact, note that $s_j$ determines a simplex in $C_{|s_j|-1}(\mathbb{\Delta}^n)$ and the condition $\text{max }s_i=\text{min }s_{i+1}$ is equivalent to $s_i \otimes s_{i+1}$ being in $C_\bullet(\mathbb{\Delta^n}) \square_{C_0(\mathbb{\Delta^n})} C_\bullet(\mathbb{\Delta^n})$. Under this identification, each map $d^0_i$ corresponds to applying the Alexander-Whitney coproduct (with appropriate grading shifts) and each $d^1_i$ corresponds to applying the map $\widetilde{\partial}$ (with appropriate grading shifts). Thus the signs above exactly correspond to the signs of each term in the differential of \ref{Qdelta}, which are determined by the Koszul sign rule. In fact, we can put $d^{0}$ and $d^{1}$ together in a formal sum
    \be
    \partial^F  = \sum\limits_{i=0}^k  \big( d^{1}_i- d^{0}_i \big),
    \ee
    so that, under the above correspondence, $\partial^F$ corresponds exactly to the differential of \ref{Qdelta}. Thus, verifying
    \be
    \partial^F \circ \partial^F  = 0
    \ee
    is analogous to checking the differential of \ref{Qdelta} squares to zero. Lastly, let
    \begin{eqnarray}\label{Fboundary}
        \partial^{F \times \mathbb{I}} = \partial^F \times \id_{I^m} - \id_{F_n}\times \partial^{\mathbb{I}},
    \end{eqnarray}
   where $\partial^{\mathbb{I}}$ is the cubical boundary map. 
   
   We now define certain codegeneracy maps between freehedra in order to normalize this construction. Observe that the vertices of any freehedron $F_n$ are labeled by expressions 
    \[ \{ s_0 | \cdots | s_{l-1}\} s_l \{ s_{l+1} | \cdots | s_k \}\] 
    where $\abs{s_i} = 2$ if $i\neq l$ and $\abs{s_l}=1$. Let $\zeta^i \colon [n] \to [n-1], i=0,...,n-1$, be the codegeneracy maps of the simplicial category, i.e.
    \begin{equation*}
        \zeta^i(j) = \begin{cases}
            j \quad \text{ if } j\leq i\\
            j-1 \quad \text{ if } j >i,
        \end{cases}
    \end{equation*}
    extended to subsets $s_j \subseteq \{0,...,n\}$ in the obvious way. We can extend these maps further to a map from the vertices of $F_n$ to the vertices of $F_{n-1}$ given by
    \[ \zeta^i(\{ s_0 | \cdots | s_{l-1}\} s_l \{ s_{l+1} | \cdots | s_k \} ):= \{ \zeta^i(s_0) | \cdots | \zeta^i(s_{l-1})\} \zeta^i(s_l) \{ \zeta^i(s_{l+1}) | \cdots | \zeta^i(s_k) \}, \]
    where if $s_j = \{ i,i+1 \}$ then $\zeta^i(s_j)=\{ i\}$ is omitted from the expression, e.g.
    \[
    \zeta^1(\{\}0\{01|12\}) = \{\}0\{01\}.
    \]
    A choice of a continuous cellular extension to the freehedra produces maps
    \[
    \zeta^i : F_n \rightarrow F_{n-1} \quad i\in 0, ... , n-1.
    \]
   The maps just defined induce $R$-linear maps
    \[
    \zeta_i : R\langle \mathsf{Top} (F_n \times I^m, X)\rangle \rightarrow R\langle \mathsf{Top} (F_{n+1} \times I^m, X) \rangle \quad i\in 0, ... , n-1
    \]
    by precomposition with $\zeta^i \times \id_I^{m}$, which are compatible with $\partial^{F \times \mathbb{I}}$. Similarly, we denote by 
    \[
    \zeta^{\mathbb{I}}_i : R\langle\mathsf{Top} (F_n \times I^m, X)\rangle \rightarrow R\langle\mathsf{Top} (F_n\times I^{m+1}, X)\rangle
    \]
    the cubical degeneracy maps. 
    \begin{definition}
    For any topological space $X$, define a dg $R$-module by setting
    \be
        C^{F \times \mathbb{I}}_r(X) := \bigoplus\limits_{n+m=r} R\langle\mathsf{Top}(F_n \times I^m, X) \rangle / \langle \text{Im}(\zeta, \zeta^{\mathbb{I}}) \rangle
    \ee
    where $F_n$ is the $n$-dimensional freehedron and $I^m$ is the $m$-dimensional cube, with differential induced by
    \[ \partial^{F \times \mathbb{I}} : C^{F \times \mathbb{I}}_\bullet(X) \rightarrow C^{F \times \mathbb{I}}_{\bullet -1}(X),  \]
   as described in \ref{Fboundary}.   
    \end{definition}
We now define a second version of the complex of normalized chains.
\begin{definition}\label{def:prismatic-chains}
 For any topological space $X$, define a dg $R$-module $(C^{\Delta \times \mathbb{I}}_\bullet(X),\partial^{\Delta \times \mathbb{I}})$ as follows. The underlying graded $R$-module is defined as
    \be
        C^{\Delta \times \mathbb{I}}_r(X) := \bigoplus\limits_{n+m=r} R\langle \mathsf{Top} (\Delta^n \times I^m, X)\rangle/ \langle \text{Im}(\zeta^\Delta, \zeta^{\mathbb{I}})\rangle
    \ee
    where $\Delta^n$ is the nth simplex, $I^m$ is the $m$-dimensional cube, and $\zeta^\Delta, \zeta^{\mathbb{I}}$ are the degree $+1$ linear maps induced by the usual simplicial and cubical degeneracy maps. The differential
    \be
        \partial^{\Delta \times \mathbb{I}} : C^{\Delta \times \mathbb{I}}_\bullet(X) \rightarrow C^{\Delta \times \mathbb{I}}_{\bullet - 1}(X)
    \ee
    is induced by the map 
    \be
        \partial^{\Delta \times \mathbb{I}} = \partial^\Delta \times \id_{I^m} - \id_{\Delta^n}\times \partial^{\mathbb{I}},
    \ee
    where $\partial^\Delta$ is the simplicial boundary map and $\partial^{\mathbb{I}}$ the cubical boundary map. 
\end{definition} 

\subsection{Models for path spaces}

We construct two maps (\ref{Gmap0} and \ref{Fmap0} below) relating the constructions of section \ref{comparingresolutions}, the two chain models described in the previous section, and path spaces. We will make use of a construction first introduced by Adams \cite{adams}. Adams argued for the existence a collection of continuous maps
\[
\{  \Theta_n \colon I^{n-1} \rightarrow \mathbf{P}(\Delta^n)(v_0,v_n)\}_{n \in \mathbb{N}} 
\] 
where $\Delta^n$ is the topological $n$-simplex and $\mathbf{P}(\Delta^n)(v_0,v_n)$ is the space of paths in $\Delta^n$ from the $0$-th vertex $v_0$ to the $n$-th vertex $v_n$, see \ref{categoriesofpaths}. These maps are required to satisfy the following compatibilities 
\begin{eqnarray} \Theta_n \circ e_j^1=P(d_j) \circ \Theta_{n-1} \text{ for } j=1, \ldots, n, \text{ and }
\end{eqnarray}
\begin{eqnarray} \Theta_n \circ e_j^0=(P(\mathfrak{f}_j) \circ \Theta_j) * (P(\mathfrak{l}_{n-j}) \circ \Theta_{n-j}) \text{ for } j=1, \ldots, n,
\end{eqnarray}
where 
$$e_j^0,e_j^1: I^{n-1} \rightarrow I^{n}$$ are the cubical co-face maps 
\[e_j^\varepsilon(t_1,\ldots,t_{n-1})=(t_1,\ldots,t_{j-1},\varepsilon,t_j,\ldots,t_{n-1}),\]

$$d_j \colon \Delta^{n-1} \to \Delta^n$$
is the linear inclusion map given by embedding $\Delta^{n-1}$ as the $(n-1)$-dimensional face of $\Delta^n$ opposite to the $j$-th vertex, and $\mathfrak{f}_j \colon \Delta^j \to \Delta^n$ and $\mathfrak{l}_{n-j} \colon \Delta^{n-j} \to \Delta^n$ denote the first and last $j$-dimensional and $(n-j)$-dimensional co-faces of $\Delta^n$, respectively.

We give an explicit construction of $\Theta_n$ for completeness. Given $y,z \in \mathbb{R}^n$ denote by
\[
\gamma (y,z) \colon [0, \abs{z-y}] \rightarrow \mathbb{R}^{n}
\]
the straight line path from $y$ to $z$ parameterized by arc-length, concretely
\[
\gamma(y,z)(s) = y + \frac{s}{\abs{z-y}}(z-y).
\]
For any $w =(w_1,...,w_{n-1}) \in I^{n-1}$ we define $p_1(w),...,p_{n-1}(w) \in \Delta^n$ inductively by
\begin{align*}
    & p_1(w)= v_0 + w_1 (v_1 - v_0),\\
    & p_j(w) = p_{j-1}(w) + w_j ( v_j - p_{j-1}(w)) .
\end{align*}
Define $l_n\colon I^{n-1} \rightarrow \mathbb{R}$ by
\[
l_n(w) = \abs{ p_1(w) -v_0 } + \abs{ p_2(w) - p_1(w) } + \cdots + \abs{ v_n - p_{n-1}(w) },
\]
and
\[
\Theta_n(w) \colon [0 , l_n(w)]\rightarrow \Delta^n
\]
as the piecewise linear path given by concatenating the straight line segments connecting the ordered sequence of points $v_0,p_1(w),...,p_{n-1}(w), v_n$, i.e.
\[
\Theta_n(w) = \gamma (v_0,p_1(w)) * \gamma(p_1(w),p_2(w)) * \cdots * \gamma(p_{n-1}(w),v_n) .
\]

For any topological space $X$, we now define a morphism of dg categories
\[
\Gamma \colon \mathbf{\Omega}(\mathcal{C}(X)) \rightarrow \mathbf{P}^{\mathbb{I}}X,
\]
where $\mathbf{P}^{\mathbb{I}}$ is defined in \ref{categoriesofpaths}, as follows. 

First, the same argument as described by Adams in \cite[p.410-411] {adams} shows that given any singular simplex $\sigma \colon \Delta^n \to X$, there is a homotopy
\[ \Gamma_t(\sigma) \colon I^{n-1} \to \mathbf{P}X(\sigma(v_0), \sigma(v_n)) \]
such that
\begin{enumerate}
\item $\Gamma_0(\sigma)= P(\sigma) \circ \Theta_n$,

\item $\text{Im } \Gamma_t(\sigma) \subseteq \mathbf{P} (\sigma(\Delta^n))(\sigma(v_0), \sigma(v_n))$,

\item If $\sigma(\Delta^n)=v$ for some $v \in X$, 
then $\Gamma_1(\sigma)(I^{n-1})$ is the constant Moore loop at $v$ parametrized by the zero length interval $\{0\}=[0,0]$. 

\item $\Gamma_t(\sigma) \circ e_j^1=\Gamma_t(\sigma \circ d_j)$ for $j=1, \ldots, n$, and 

\item  $\Gamma_t(\sigma)  \circ e_j^0=\Gamma_t(\mathfrak{f}_j \circ \sigma) * \Gamma_t(\mathfrak{l}_{n-j} \circ \sigma)$ for  $j=1, \ldots, n$.

\end{enumerate}

The two dg categories have the same objects so the map between objects is defined to be the identity. Given any morphism \[\sigma= \{ \sigma _1 | \ldots | \sigma_k \} \in \mathbf{\Omega}(\mathcal{C}(X))(x,y),\] where each $\sigma_i \colon \Delta^{n_i} \to X$ is a singular simplex representing a normalized class in $\mathcal{C}_{n_i}(X)$,  define
\[\Gamma(\sigma) = \Gamma_1(\sigma_1) * \cdots * \Gamma_1(\sigma_k) \in \mathbf{P}(X)(x,y). \]
Using the deformation of the maps $\Theta_n$ given by $\Gamma_1$ when defining $\Gamma$, ensures that 
\[ \Gamma \colon \mathbf{\Omega}(\mathcal{C}(X))(x,y) \to \mathbf{P}(X)(x,y)\] is a chain map for every $X$ and $x,y\in X$. 

A fundamental result, extending a classical theorem of Adams \cite{adams}, is the following many-object version of a theorem of Rivera and Zeinalian \cite{rivera2018cubical}. See also \cite{riverasaneblidzepath, MRZ, holstein-lazarev}.

\begin{theorem} \label{adamsthm}
    For any topological space $X$, $\Gamma \colon \mathbf{\Omega}(\mathcal{C}(X)) \to \mathbf{P}^{\mathbb{I}}X$ is a quasi-equivalence of dg categories. Furthermore, $\Gamma$ is natural at the level of categories enriched over graded modules upon passing to homology.
\end{theorem}
\noindent \textit{Sketch of Proof.} The theorem boils down to proving that, for any $x,y \in X$, the map $\Gamma$ induces a quasi-isomorphism of chain complexes \[ \Gamma_{x,y} \colon \mathbf{\Omega}(\mathcal{C}(X))(x,y) \to \mathbf{P}^{\mathbb{I}}X(x,y).\]
The chain complex $\mathbf{\Omega}(\mathcal{C}(X))(x,y)$ is \textit{isomorphic} to the normalized cubical chains on a \textit{cubical set (with connections)} $\mathbf{\Omega}^{\mathbb{I}}X(x,y)$. The chain map $\Gamma_{x,y}$ is induced by a continuous map of spaces 
\[ \omega_{x,y} \colon |\mathbf{\Omega}^{\mathbb{I}}X(x,y)| \to \mathbf{P}X(x,y).\] The map $\omega_{x,y}$ can be shown to be a weak homotopy equivalence by constructing a map between two quasi-fibrations over $X$ with contractible total space for which  $\omega_{x,y}$ sits as the map between the fibers. We refer to \cite[Theorem 1]{riverasaneblidzepath} for details. Naturality with respect to continuous maps of topological spaces boils down to relating different parametrizations of the families of constant loops through appropriate homotopies. \qed
\\

For simplicity, from now on denote  \[\mathcal{A}(X)=\mathcal{M}( \mathbf{\Omega}(\mathcal{C}(X)))\] considered as a monoid in the monoidal category $(\mathcal{C}_0(X)\text{-}\mathsf{biComod}, \square)$. Similarly, denote \[\mathcal{P}(X) = \mathcal{M}(\mathbf{P}^{\mathbb{I}}(X))\]
also considered as a monoid in $(\mathcal{C}_0(X)\text{-}\mathsf{biComod}, \square)$ via concatenation of paths. The map discussed above may be regarded as a monoid map
\[\Gamma \colon \mathcal{A}(X) \to \mathcal{P}(X)\]
in $(\mathcal{C}_0(X)\text{-}\mathsf{biComod},\square)$. We will also consider the topological space
\[
    P^*X = \{ (\gamma,l,z) \in PX \times [0,\infty)] | (\gamma,l) \in PX \text{ and } z \in [0,l]\}.
\]
For any $\gamma = (\gamma,l,z) \in P^*X$ we call $\gamma(z)$ the \textit{marked point} of $\gamma$. We denote \[P^*X(x,y)= \{(\gamma,l,z) \in P^*X | \gamma(0)=x, \gamma(l)=y\}.\]

 The dg $R$-modules  $C_\bullet^{F \times \mathbb{I}}(P^*X)$ and $C_\bullet^{\Delta \times \mathbb{I}}(P^*X)$ have natural $\mathcal{C}_0(X)$-bicomodules given by evaluating a path at its starting point or end point. Furthermore, these have a natural $\mathcal{P}(X)$-bimodule structure given by concatenation of paths. Hence, for every $X$, the map $\Gamma$ induces dg $\mathcal{A}(X)$-bimodule structures on these two $\mathcal{C}_0(X)$-bicomodules. 

We now construct a map of dg  $\mathcal{A}(X)$-bimodules
\begin{eqnarray}\label{Gmap0}
\mathcal{G}:\mathcal{B}(\mathcal{A}(X),\mathcal{A}(X),\mathcal{A}(X)) \rightarrow \mathcal{P}(X)\square {C_\bullet^{\Delta \times \mathbb{I}}(P^*X)} \square \mathcal{P}(X).
\end{eqnarray}
On any generator $ a_0 [a_1 \vert \dots \vert a_n]a_{n+1}$ in $\mathcal{B}(\mathcal{A}(X),\mathcal{A}(X),\mathcal{A}(X))$, this map is given by
\[\mathcal{G}(a_0[a_1 | \cdots | a_n]a_{n+1})= \Gamma(a_0) \square \widetilde{\mathcal{G}}([a_1 | \cdots | a_n]) \square \Gamma(a_{n+1}),\]
where, on any $u=(u_0,\dots,u_n) $ in $\Delta^n$ (given in barycentric coordinates) and $w=(w_1, \dots, w_{n}) $ in $  I^{|a_1|} \times \cdots \times I^{|a_{n}|}$, we have
\[\widetilde{\mathcal{G}}([a_1 \vert \dots \vert a_n])(u,w) = \left( \gamma_1 * \cdots *\gamma_{n}  ,  \sum\limits_{j=1}^{n} m_j - \sum\limits_{0 \leq i < j \leq n} u_im_j \right),\]
for $\gamma_j=(\gamma_j, m_j)=\Gamma(a_j)(w_j)$. The map $\mathcal{G}$ is a version of the map $\lambda$ defined in \cite[Section V.1]{Goodwillie}.

We also construct a map of dg  $\mathcal{A}(X)$-bimodules
\begin{eqnarray}\label{Fmap0}
    \mathcal{F}:\mathcal{Q}(\mathcal{A}(X), \mathcal{C}(X),\mathcal{A}(X)) \rightarrow \mathcal{P}(X)\square {C_\bullet^{F \times \mathbb{I}}(P^*X)} \square \mathcal{P}(X).
    \end{eqnarray}
On any generator $a\square c \square b $ in $\mathcal{Q}(\mathcal{A}(X), \mathcal{C}(X),\mathcal{A}(X))$, this map is given by 
\[ \mathcal{F}(a \square c \square b)= 
\Gamma(a) \square \widetilde{\mathcal{F}}(c) \square \Gamma(b),
\]
where $\widetilde{\mathcal{F}} (c) \colon F_{|c|} \to P^*X$ is defined as follows.  Since $F_{\abs{c}}$ is a subdivided cube, we may describe a point in $F_{\abs{c}}$ as a point in the cube $I^{\abs{c}}$, say $(t,u)=(t,u_1, \dots, u_{n-1})$. Then 
\be
\widetilde{\mathcal{F}}(c)(t,u) = \left( \gamma, m t  \right),
\ee
where $(\gamma,m) = \Gamma(c)(u)$. The map $\mathcal{F}$ is a version of the map $\Upsilon$ constructed in \cite[Theorem 2]{rivera-saneblidze}. Routine checks yield that $\mathcal{G}$ and $\mathcal{F}$ are chain maps. 

\subsection{A geometric model for the map $\alpha$}
We model the map
\[\alpha \colon \mathcal{Q}(\mathcal{A}(X), \mathcal{C}(X),\mathcal{A}(X)) \to \mathcal{B}(\mathcal{A}(X), \mathcal{A}(X),\mathcal{A}(X)) \]
of \ref{def:alpha-pi} in terms of a natural map \[\Upsilon \colon C^{F \times \mathbb{I}}_\bullet (P^*X) \to C^{\Delta \times \mathbb{I}}_\bullet(P^*X)\]
arising from the combinatorics of the Goodwillie polytopes. The maps $\alpha$ and $\Upsilon$ will be compatible via $\mathcal{F}$ and $\mathcal{G}$.

The map $\Upsilon$ will be induced by an explicit continuous map $\upsilon: G_n \rightarrow F_n$ inducing the map of abstract polytopes $\rho: \fancy{G}_n \rightarrow \fancy{F}_n $ described in \ref{def:rho}, which we now describe. The subdivision of $G_n$ into prisms described in section \ref{gpolytopes} is indexed by all subsets $S$ of $\{2, ..., n\}$, where
\[
G_n^S \iso \Delta^{m} \times I^{k},
\]
$m=\abs{S}+1$, and $k=n-m$. We define a family of maps 
\[
\upsilon_{S} : \Delta^{\abs{S}+1} \times I^{n-(\abs{S}+1)} \rightarrow F_n
\] 
that ``glue together'' and determine $\upsilon$.

Recall that each $\Theta_{n}$ (defined in previous section) has an associated continuous function $l_n \colon I^{n-1} \to \mathbb{R}$ giving the length of the parameterizing interval of the corresponding path. Writing $S = \{s_1 ,..., s_r\} \subset \{2,...,n\}$ with $s_1< \cdots< s_r$ we obtain a partition  $\{1,...,s_1 \},\{ s_1+1,...,s_2\},...,\{s_{r-1}+1,...,s_r \},\{s_r+1,...,n\}$ of $\{1,...,n\}$. Let $n_1,...,n_{\abs{S}+1}$ be the cardinality of the sets in this partition, respectively. 
If $(u,w) \in \Delta^{m} \times I^{k}$, where $u=(u_0,...,u_m)$ is given in barycentric coordinates and $w=(w_1,...,w_k)$, define
\[
\upsilon_{S}(u,w) = \left( \frac{1}{\sum\limits_{j=1}^{m} l_{n_j}(w_j)} \left( \sum\limits_{j=1}^{m} l_{n_j}(w_j) ~ - \sum\limits_{0\leq i <j 
\leq m} u_i ~ l_{n_j}(w_j) \right) , w^S \right),
\]
where $w^S$ is the image of $w$ under the embedding of $I^k$ into $I^n$ that inserts a zero in the $i$-th coordinate for every $i+1$ in $S$. Concretely, if $w=(w_1,...,w_k)$, and $w^S=(w_1^S,...,w_n^S)$, then 
\begin{equation}\label{def:w_S} 
w_i^S = \begin{cases}
    0 \quad \text{ if } i+1 \in S,\\
    w_j \quad  j=i~ - \sum\limits_{\substack{j \leq i \\ j \in S}} 1, \text{ otherwise.}
\end{cases}
\end{equation}
\begin{definition}
For any topological space $Y$, we define a natural chain map
    \[
    \Upsilon \colon C^{F \times \mathbb{I}}_\bullet(Y) \rightarrow C^{\Delta \times \mathbb{I}}_\bullet(Y)
    \]
  by declaring
    \[
    \Upsilon(\sigma ) = \sum\limits_{S \subset \{2,...,n\}}  \Upsilon_S(\sigma)  :=\sum\limits_{S \subset \{2,...,n\}}  \sigma \circ (\upsilon_S \times \id_{I^m} ),
    \]
     for any $\sigma: F^n \times I^m \rightarrow   Y$.
\end{definition}
\begin{theorem} \label{thmresolutions} For any topological space $X$, the following is a diagram of quasi-isomorphisms of dg $\mathcal{A}(X)$-bimodules that commutes up to chain homotopy:
\[\begin{tikzcd}
 {\mathcal{Q}(\mathcal{A}(X), \mathcal{C}(X),\mathcal{A}(X)) } & {\mathcal{P}(X) \square C^{F \times \mathbb{I}}_\bullet(P^*X) \square \mathcal{P}(X)} \\
	{\mathcal{B}(\mathcal{A}(X),\mathcal{A}(X),\mathcal{A}(X))} & {\mathcal{P}(X) \square C_\bullet^{\Delta \times \mathbb{I}}(P^*X) \square \mathcal{P}(X) } 
	\arrow["\alpha"', from=1-1, to=2-1]
	\arrow["{\mathcal{F}}", from=1-1, to=1-2]
	\arrow["{\mathcal{G}}"', from=2-1, to=2-2]
	\arrow["\Upsilon", from=1-2, to=2-2]
\end{tikzcd}\]
\end{theorem}
\begin{proof}
    Let $ a \square c \square b $ be a generator of $\mathcal{Q}(\mathcal{A}(X), \mathcal{C}(X),\mathcal{A}(X)) $. 
     We begin by describing a correspondence between terms in the sums defining $\mathcal{G}\alpha$ and $\Upsilon \mathcal{F}$. First, observe that the terms of $\mathcal{G}\alpha(a \square c \square b)$ are given by all possible iterations of the Alexander-Whitney coproduct applied to $c \in \mathcal{C}_n(X)$, and hence can be labeled by all possible expressions of the form 
     \[0 \cdots s_1 | s_1 \cdots s_2| \cdots | s_{r-1} \cdots s_r | s_r \cdots n\] 
     with $0<s_1< \cdots < s_r < n$. These labels are in bijection with the subsets of $\{2,\dots , n \}$ through 
     \[0 \cdots s_1 | s_1 \cdots s_2| \cdots | s_{r-1} \cdots s_r | s_r \cdots n \longleftrightarrow \{ s_1 + 1, s_2+1, \ldots, s_r+1 \} .\] 
     Denote the summand of $\mathcal{G}\alpha(a \square c \square b)$ labeled by the set $S \subset \{2,...,n\}$ under the above bijection by $\mathcal{G}\alpha(a \square c \square b)_S$. 
     Concretely, say $S= \{ s_1 +1 < s_2+1 < ... < s_r +1  \}$. Then the associated summand is
     \[
     \mathcal{G}\alpha(a \square c \square b)_S = \mathcal{G}(a [ \{c' \} | \cdots | \{ c^{(r)} \} | \{ c^{(r+1)} \}  ]b)
     = \Gamma(a) \square \widetilde{\mathcal{G}}([ \{c' \} | \cdots | \{ c^{(r+1)} \}  ])\square \Gamma(b)
     \]
     where $\abs{c'} = s_1$, $\abs{c^{(j)}} = \abs{\{ s_{j-1} , s_{j-1} + 1 ,... , s_{j} \}}$ for $1<j\leq r$, and $\abs{c^{(r+1)}}= n - s_r$ are the degrees of the $c^{(j)}$'s in $\mathcal{C}(X)$. Let $n_j = \abs{c^{(j)}}-1$ for all $j$, and denote
    \[
    \widetilde{\mathcal{G}}_S(c) = \widetilde{\mathcal{G}}([ \{c' \} | \cdots | \{ c^{(r+1)} \}  ]) \colon \Delta^{r+1} \times I^{n_1 -1} \times \cdots \times I^{n_{r+1}-1} \rightarrow P^*X 
    \]
    with $n_1 + \cdots + n_{r+1} = n$.
     %
     By definition, the summands of $\Upsilon$ are indexed by the sets $S\subset \{  2,...,n \}$, and
     \[
     \Upsilon_S(\mathcal{F}(a\square c \square b))=\Upsilon_S(\Gamma(a)\square \widetilde{\mathcal{F}}(c) \square \Gamma(b))
     = \Gamma(a)\square \Upsilon_S(\widetilde{\mathcal{F}}(c)) \square \Gamma(b),
     \]
     where
     \[
     \Upsilon_S(\widetilde{\mathcal{F}}(c)) \colon \Delta^{r+1} \times I^{n_1 -1} \times \cdots \times I^{n_{r+1}-1} \rightarrow P^*X.
     \]
     We argue that the above square commutes up to reparametrization of paths, and that, more precisely, the homotopy $\Gamma_t(c)$ (defined in the previous section) induces a chain homotopy between $\Upsilon \mathcal{F}$
     and $\mathcal{G}\alpha$.
     
     We proceed to compute both 
     $\widetilde{\mathcal{G}}_S(c)$ and $\Upsilon_S \widetilde{\mathcal{F}}(c)$
     for every $S \subset \{2,...,n\}$.
    If $u=(u_0,...,u_{r+1}) \in \Delta^{r+1}$ is given in barycentric coordinates, and $w=(w_1,...,w_{r+1}) \in I^{n_1 -1} \times \cdots \times I^{n_{r+1}-1}$, we have
\begin{align}\label{m_j}
     \widetilde{\mathcal{G}}_S(c)(u,w)= 
         \left( \gamma_1 * \cdots * \gamma_{r+1} ,
          \sum\limits_{j=1}^{r+1} m_j - \sum\limits_{0\leq i<j\leq r+1} u_i m_j     
         \right)
\end{align}
where $(\gamma_j,m_j)=\Gamma(c^{(j)})(w_j)$, $j=1,...,r+1$.

On the other hand,
    \[
    \Upsilon_S(\widetilde{\mathcal{F}}(c)) (u,w)=
             \widetilde{\mathcal{F} } (c )\left( \frac{1}{\sum\limits_{j=1}^{r+1} l_{n_j}(w_j)} \left( \sum\limits_{j=1}^{r+1} l_{n_j}(w_j)-   \sum\limits_{0\leq i<j\leq r+1} u_i l_{n_j}(w_j)      \right) , w^S  \right)
    \]
     where $w^S$ is as defined in \ref{def:w_S}. Since
     \[\Gamma(c)(w^S) = (\gamma_1 * ... * \gamma_{r+1}, \sum_{j=1}^{r+1}m_j),\] 
     by the definition of $\widetilde{\mathcal{F}}(c)$ we have that $\Upsilon_S(\widetilde{\mathcal{F}}(c)) (u,w) =$
     \begin{align} \label{l_{n_j}}
              \left(   \gamma_1 * ... * \gamma_{r+1} , \frac{\sum_{j=1}^{r+1}m_j}{\sum\limits_{j=1}^{r+1} l_{n_j}(w_j)} \left( \sum\limits_{j=1}^{r+1} l_{n_j}(w_j)-   \sum\limits_{0\leq i<j\leq r+1} u_i l_{n_j}(w_j)      \right)        \right).
     \end{align}
     
The homotopy $\Gamma_t(c)$ reparametrizes the paths $(\gamma_j,m_j)$ into new paths $(\gamma'_j, l_{n_j}(w_j))$. While all of $l_{n_j}(w_j)$ are non-zero, some $m_j$ might be zero corresponding to a constant loop. 

This yields a chain homotopy that deforms each term  \ref{l_{n_j}} into
   \begin{align} 
              \left(   \gamma'_1 * ... * \gamma'_{r+1} ,  \sum\limits_{j=1}^{r+1} l_{n_j}(w_j)-   \sum\limits_{0\leq i<j\leq r+1} u_i l_{n_j}(w_j)            \right),
     \end{align}
a reparametrization of \ref{m_j}. This shows that the diagram commutes up to a chain homotopy. 

From Proposition \ref{chaincontraction}, it follows that $\alpha$ is a quasi-isomorphism. A standard acyclic models argument shows that  $\Upsilon$ is a quasi-isomorphism. 
     
     We now argue that $\mathcal{G}$ is a quasi-isomorphism. This forces $\mathcal{F}$ to be a quasi-isomorphism as well, which will conclude the proof. Consider the following diagram  
\[\begin{tikzcd}
	{\mathcal{B}(\mathcal{A}(X),\mathcal{A}(X),\mathcal{A}(X))} & {\mathcal{P}(X) \square C_\bullet^{\Delta \times \mathbb{I}}(P^*X) \square \mathcal{P}(X)} & {\mathcal{P}(X) \square C_\bullet^{\mathbb{I}}(P^*X) \square \mathcal{P}(X)} \\
	{\mathcal{A}(X)} && {\mathcal{P}(X)}
	\arrow["{\mathcal{G}}", from=1-1, to=1-2]
	\arrow["\simeq"', from=1-1, to=2-1]
	\arrow["\Gamma"', from=2-1, to=2-3]
	\arrow["\iota"', hook', from=1-3, to=1-2]
	\arrow["\simeq", from=1-3, to=2-3].
\end{tikzcd}\]
     The left vertical map is the composition \[\mathcal{B}(\mathcal{A}(X),\mathcal{A}(X),\mathcal{A}(X)) \twoheadrightarrow \mathcal{A}(X) \underset{\mathcal{C}_0(X)}{\square} \mathcal{A}(X) \to \mathcal{A}(X),\] of the natural projection map and the product of the monoid structure. This composition of maps is a quasi-isomorphism since the bar construction is a resolution of $\mathcal{A}(X)$ as an $\mathcal{A}(X)$-bimodule. In fact, a contracting homotopy can be constructed by using the unit map $\mathcal{C}_0(X) \to \mathcal{A}(X)$, as usual. The right vertical map is the map on cubical chains induced by forgetting the marked point of a marked path, which is clearly a homotopy equivalence at the space level, followed by concatenation of paths. The inclusion map $\iota$ is clearly a quasi-isomorphism by acyclic models and $\Gamma$ is a quasi-isomorphism by \ref{adamsthm}. This implies $\mathcal{G}$ is a quasi-isomorphism. 
     \end{proof}


\subsection{Models for the free loop space}
Let us define a chain map \[g^F \colon \big(\mathcal{P}(X) \square C^{F \times \mathbb{I}}_\bullet(P^*X)\square \mathcal{P}(X) \big) \otimes_{\mathcal{P}(X)^e}  \mathcal{P}(X)  \to C^{F \times \mathbb{I}}_\bullet(\mathcal{L}X).\] The chain complex $\mathcal{P}(X) \square C^{F \times \mathbb{I}}_\bullet(P^*X)\square \mathcal{P}(X) $ is generated by pairs $\sigma \otimes \kappa$ where $\sigma \colon F_n \times I^m 
\to P^*X(x,y)$ and $\kappa \colon I^k \to \mathbf{P}X(y,x)$. For any such generator, define \[g^F(\sigma \otimes \kappa) \colon F_n \times I^{m} \times I^k \to \mathcal{L}X\] on any $(z,v,w) \in F_n \times I^m \times I^k$ by declaring $g^F(\sigma \otimes \kappa)(z,v,w)$ to be the loop that starts at the marked point (remembering the time parameter indicating when it occurs) of $\sigma(z,v)$ and runs the path  $\sigma(z,v)$ until it reaches $y \in X$, then continues with the path $\kappa(w)$ from $y$ to $x$ and ends by running $\sigma(z,v)$ again from $x$ to its marked point. Similarly, we may define a map \[g^{\Delta} \colon \big(\mathcal{P}(X) \square C^{\Delta \times \mathbb{I}}_\bullet(P^*X)\square \mathcal{P}(X) \big)\otimes_{\mathcal{P}(X)^e}  \mathcal{P}(X) \to C^{\Delta \times \mathbb{I}}_\bullet(\mathcal{L}X).\]

Denote by \[\theta^F\colon \big( \mathcal{A}(X) \square C^{F \times \mathbb{I}}_\bullet(P^*X)  \square \mathcal{A}(X) \big) \otimes_{\mathcal{A}(X)^e} \mathcal{A}(X) \to C^{F \times \mathbb{I}}_\bullet(\mathcal{L}X)\] 
and
\[\theta^\Delta \colon \big( \mathcal{A}(X) \square C^{\Delta \times \mathbb{I}}_\bullet(P^*X)  \square \mathcal{A}(X) \big) \otimes_{\mathcal{A}(X)^e} \mathcal{A}(X) \to C^{\Delta \times \mathbb{I}}_\bullet(\mathcal{L}X),\]
the maps obtained by pre-composing $g^F$ and $g^\Delta$, respectively, with Adams' map $\Gamma \colon \mathcal{A}(X) \to \mathcal{P}(X)$.

\begin{theorem} \label{freeloopsmodel}
    For any topological space $X$, the following is a diagram of quasi-isomorphisms of chain complexes that commutes up to chain homotopy

\[\begin{tikzcd} \label{hochdiag}
	{\textnormal{coHoch}_\bullet(\mathcal{C}(X))} && {\big( \mathcal{A}(X) \square C^{F \times \mathbb{I}}_\bullet(P^*X)  \square \mathcal{A}(X) \big) \otimes_{\mathcal{A}(X)^e} \mathcal{A}(X)} & {C_\bullet^{F \times \mathbb{I}}(\mathcal{L}X)} \\
	{\textnormal{Hoch}_\bullet(\mathbf{\Omega}(\mathcal{C}(X)))} && {\big( \mathcal{A}(X) \square C^{\Delta \times \mathbb{I}}_\bullet(P^*X)  \square \mathcal{A}(X) \big) \otimes_{\mathcal{A}(X)^e} \mathcal{A}(X)} & {C_\bullet^{\Delta \times \mathbb{I}}(\mathcal{L}X)}
	\arrow["\widetilde{\mathcal{F}} \underset{\mathcal{A}(X)^e}{\otimes}{id_{\mathcal{A}(X)}}"', from=1-1, to=1-3]
	\arrow["{\overline{\alpha}}"', from=1-1, to=2-1]
	\arrow["{\widetilde{\mathcal{G}} \underset{\mathcal{A}(X)^e}{\otimes}{id_{\mathcal{A}(X)}}}"', from=2-1, to=2-3]
	\arrow["{\theta^{\Delta}}", from=2-3, to=2-4]
	\arrow["{\Upsilon }", from=1-4, to=2-4]
	\arrow["{\Upsilon \underset{\mathcal{A}(X)^e}{\otimes}{id_{\mathcal{A}(X)}}}", from=1-3, to=2-3]
	\arrow["{\theta^F}", from=1-3, to=1-4].
\end{tikzcd}\]
In particular, we have quasi-isomorphism
\begin{eqnarray}\label{coHochmodel}
    \textnormal{coHoch}_\bullet(\mathcal{C}(X)) \xrightarrow{\Upsilon \circ \theta^F \circ  (\widetilde{\mathcal{F}} \underset{\mathcal{A}(X)^e}{\otimes}{id_{\mathcal{A}(X)}}) } C_\bullet^{\Delta \times \mathbb{I}}(\fancy{L}X) \xrightarrow{\mathcal{E}} C_\bullet(\fancy{L}X),
    \end{eqnarray}
where $C_\bullet(\mathcal{L}X)$ denotes the normalized (simplicial) singular chains on $\mathcal{L}X$ and $\mathcal{E}$ is the quasi-isomorphism induced by the Eilenberg-Zilber subdivision operator. This quasi-isomorphism is natural on homology with respect to continuous maps of spaces.
\end{theorem}
\begin{proof}
Consider the diagram
\[\begin{tikzcd}[column sep=10em]
 {\mathcal{Q}(\mathcal{A}(X), \mathcal{C}(X),\mathcal{A}(X)) } & {\mathcal{A}(X) \square C^{F \times \mathbb{I}}_\bullet(P^*X) \square \mathcal{A}(X)} \\
	{\mathcal{B}(\mathcal{A}(X),\mathcal{A}(X),\mathcal{A}(X))} & {\mathcal{A}(X) \square C_\bullet^{\Delta \times \mathbb{I}}(P^*X) \square \mathcal{A}(X) } 
	\arrow["\alpha"', from=1-1, to=2-1]
	\arrow["\textnormal{id}_{\mathcal{A}(X)} \square \widetilde{\mathcal{F}} \square \textnormal{id}_{\mathcal{A}(X)}", from=1-1, to=1-2]
	\arrow[" \textnormal{id}_{\mathcal{A}(X)} \square \widetilde{\mathcal{G}} \square \textnormal{id}_{\mathcal{A}(X)}"', from=2-1, to=2-2]
	\arrow["\Upsilon", from=1-2, to=2-2]
\end{tikzcd}\]
A similar argument to the proof of Theorem \ref{thmresolutions} shows this is a square of quasi-isomorphisms of dg (free) $\mathcal{A}(X)$-bimodules that commutes up to chain homotopy. Applying $(-) \otimes_{ \mathcal{A}(X)^e } \mathcal{A}(X)$ to this square gives the desired (left) square in the theorem.

The commutativity of the right square is easy to verify. The fact that the chain map $\mathcal{E}$ (which subdivides the cube factor into simplices) is a quasi-isomorphism follows from a standard acyclic models argument. The map
\[  \textnormal{Hoch}_\bullet(\mathbf{\Omega}(\mathcal{C}(X))) \xrightarrow{\theta^\Delta \circ  (\widetilde{\mathcal{G}}\underset{\mathcal{A}(X)^e}{\otimes}{id_{\mathcal{A}(X)}}) } C_\bullet^{\Delta \times \mathbb{I}}(\fancy{L}X) \xrightarrow{\mathcal{E}} C_\bullet(\fancy{L}X)\]
is induced by (a ``many-object" version of) the map denoted by $\lambda$ in \cite[Section V.1]{Goodwillie}, which is shown to be a quasi-isomorphism in that reference. This implies $\theta^{\Delta}$ is a quasi-isomorphism and consequently $\theta^F$ as well. Alternatively, one can also show $g^{\Delta}$ is a quasi-isomorphism directly, by observing that $g^{\Delta}$ is induced by a space level weak homotopy equivalence, which gives another proof of Goodwillie's theorem about $\lambda$ inducing a quasi-isomorphism. 
\end{proof}
We denote the composition of maps in the top row of the above diagram by
\begin{eqnarray} \label{Fmaphoch}
\mathfrak{F} \colon \text{coHoch}_\bullet(\mathcal{C}(X)) \to C_\bullet^{F \times \mathbb{I}}(\mathcal{L}X)
\end{eqnarray}
and the composition of maps in the row below by
\begin{eqnarray} \label{Gmaphoch}
\mathfrak{G} \colon \textnormal{Hoch}_\bullet(\mathbf{\Omega}(\mathcal{C}(X))) \to C_\bullet^{\Delta \times \mathbb{I}}(\mathcal{L}X).
\end{eqnarray}

\section{Cyclic homology and equivariant homology} \label{cyclichomology}
\subsection{Mixed complexes and cyclic objects} \label{mixedcomplexes}
Mixed complexes, first introduced by Kassel \cite{KASSEL1987195}, provide a suitable framework to study cyclic homology. By a \textit{mixed complex} $(M,b,B)$ we will mean a non-negatively graded $R$-module $M$, together with a degree $-1$ linear map $b$ and a degree$+1$ linear map $B$ such that 
 \begin{align}
     &b^2 = 0,\\ 
     &B^2 = 0 \text{, and}\\ 
     &bB + Bb = 0.
 \end{align}
A \textit{morphism of mixed complexes} $f \colon (M,b,B) \rightarrow (M',b',B')$ is an $R$-linear map that is compatible with the differentials $b,b',B,B'$. It is called a \textit{quasi-isomorphism} if it induces an isomorphism in homology $$H_\bullet(M, b) \to H_\bullet(M', b') .$$ 

Cyclic objects give rise to examples of mixed complexes. Recall that a \textit{cyclic object} in a category $\fancy{C}$ consists of a simplicial object $X_\bullet \colon \Delta^{\text{op}} \to \fancy{C}$ with the additional structure of an action of the cyclic group of order $n+1$ on $X_n$ for each $n\geq 0$ satisfying the following relations. Denoting by $t=t_n$  the generator of the corresponding cyclic group we require
\begin{align*}
    &d_i t = \begin{cases}
                        t d_{i-1} \quad 0<i\leq n\\
                        d_n \quad  i=0
                    \end{cases}
                    \\
    &s_{i}t = \begin{cases}
        t s_{i-1} \quad 0<i\leq n\\
        t^2 s_n \quad i=0
    \end{cases}
    \\
    &t^{n+1} = \text{id}
\end{align*}
Equivalently, one may describe a cyclic object in $\fancy{C}$ as a functor $\Lambda^{\text{op}} \rightarrow \fancy{C}$,
where $\Lambda$ is Connes' \textit{cyclic category}, see \cite{MR0777584} for a detailed description. 

The \textit{extra degeneracies} of a cyclic object $X$ are defined as the morphisms
\[
s= s_{n+1} = (-1)^{n+1} t_{n+1} s_n: X_n \rightarrow X_{n+1}.
\]
These satisfy
\begin{align*}
    &d_0s=\text{id},\\
    &d_is=sd_{i-1} \text{ for } 0 < i \leq n 
   \end{align*}
In general, $d_{n+1}s \neq sd_n$.

Any cyclic dg $R$-module (a cyclic object in the category $\mathsf{Ch}_R$) gives rise to a mixed complex upon passing to the normalized total complex. In particular, any cyclic topological space (i.e. cyclic object in the category $\mathsf{Top}$) gives rise to a mixed complex through the normalized singular chains functor. We discuss this particular case in more detail. If \[ [n] \mapsto X_n\] is a cyclic topological space, then the normalized singular chains functor (simplicial, cubical, or any version) applied to each space $X_n 
 $, yields a cyclic dg $R$-module \[[n] \mapsto (C_\bullet(X_n),\partial)\] with structure maps $C_\bullet(d_i),C_\bullet(s_j)$, and $C_\bullet(t_n)$ induced by the structure maps of $X_\bullet$. Taking the alternating sums of the face maps of the cyclic structure, we obtain an induced chain map on normalized chains 
\[
d = \sum\limits_{i=0}^n (-1)^i C_\bullet(d_i): C_\bullet (X_n) \rightarrow C_\bullet (X_{n-1})
\]
satisfying $d^2=0$. This may be regarded as a bi-complex that on bi-degree $i,j$ is given by $C_i(X_j)$. We denote its normalized total complex by $\mathbf{M}X$. By definition, the differential of $\mathbf{M}X$ is given by
\[
b = d + \tilde{\partial},
\]
where $\tilde{\partial}(\alpha) = (-1)^{k} \partial(\alpha)$ for any homogenous element $\alpha \in C_k(X_n).$ The dg $R$-module $\mathbf{M}X$ has an additional operator 
\[
B = s\circ (\sum_{i=0}^n ((-1)^nC_{\bullet}(t))^i): \mathbf{M}X_n \rightarrow \mathbf{M}X_{n+1}
\]
encoding the cyclic group actions, where $s$ is the extra degeneracy. A routine check yields that the total complex, together with the degree $-1$ map $b$, and the degree $+1$ map B define a mixed complex $(\mathbf{M}X,b,B)$, and this construction is functorial. We refer to \cite{loday2013cyclic} for further details. 

The above passage from the cyclic dg $R$-module $[n] \mapsto (C_\bullet(X_n),\partial)$ to the mixed complex $(\mathbf{M}X,b,B)$ can be defined in general yielding a normalized total complex functor from cyclic dg $R$-modules to mixed complexes. 

\subsection{Chains on the free loop space as a mixed complex} \label{chainsonLX}

Any topological space $Y$ equipped with an $S^1$-action $r \colon S^1 \times Y \to Y$ gives rise to a mixed complex. More precisely, let $\iota \colon \Delta^1 \to S^1$ be the singular $1$-cycle determined by the continuous map that identifies the two endpoints of $\Delta^1$ so that the homology class $[\iota]$ generates $H_1(S^1) \cong R\langle [\iota] \rangle$. Consider the map of degree $+1$ on normalized singular chains given by the composition
\[\mathcal{R}: C_\bullet(Y) \xrightarrow{\iota \otimes -} C_1(S^1)\otimes C_\bullet(Y) \xrightarrow{\mathcal{S}} C_{\bullet+1}(S^1 \times Y) \xrightarrow{C_{\bullet}(r)} C_{\bullet+1}(Y),\]
where $\mathcal{S}$ denotes an appropriate natural subdivision operator, which we describe below.  It follows that $\mathcal{R}^2=0$ and $\mathcal{R} \partial + \partial \mathcal{R} =0$, where $\partial$ is the boundary differential of the normalized singular chains complex. Thus $(C_\bullet(Y),\partial, \mathcal{R})$ is a mixed complex. For any singular $n$-simplex $\sigma \colon \Delta^n \to Y$, the subdivision operator is defined by
\begin{equation} \label{subdivision} \mathcal{S}( \iota \otimes \sigma ) = \sum_{i=0}^{n} (-1)^{ni} (\iota \times \sigma)|_{{\Delta}^{n+1}_i} \in C_{n+1}(S^1 \times Y),
\end{equation}
where \[\Delta^{n+1}_i = \{(u_0, \ldots, u_n, \phi) \in \Delta^n \times S^1 | \sum_{j=n-i+1}^n u_j \leq \phi \leq \sum_{j=n-i}^n u_j\}.\] Note we have used barycentric coordinates $(u_0, \ldots, u_n)$ on $\Delta^n$. The above formula determines a chain map $\mathcal{S} \colon C_\bullet(S^1)\otimes C_\bullet(Y) \to C_\bullet(S^1 \times Y)$.

In particular, the free loop space $Y= \fancy{L}X$ equipped with the $S^1$-action given by rotation of loops gives rise to a mixed complex $(C_\bullet(\fancy{L}X), \partial, \mathcal{R})$. The same construction as above gives rise to a mixed complex $(C_\bullet^{\Delta \times \mathbb{I}}(\fancy{L}X), \partial, \mathcal{R})$, where $C_\bullet^{\Delta \times \mathbb{I}}(\fancy{L}X)$ is the normalized chain complex generated by singular prisms $\Delta^n \times I^m \to \fancy{L}X$. Furthermore, the Eilenberg-Zilber map $\mathcal{E} \colon C_\bullet^{\Delta \times \mathbb{I}}(\fancy{L}X) \to C_\bullet(\mathcal{L}X)$ is a quasi-isomorphism of mixed complexes. 

For any fixed non-negative integer $n$ we shall denote by $C^{\Delta \times \mathbb{I}}_{n,\bullet}(\fancy{L}X)$ the dg $R$-module obtained by considering the graded $R$-module generated by continuous maps $\Delta^n \times I^m \to \fancy{L}X$ and then modding out by \textit{cubical} degeneracies. The differential \[\partial^{\mathbb{I}} \colon C^{\Delta \times \mathbb{I}}_{n,\bullet}(\fancy{L}X) \to C^{\Delta \times \mathbb{I}}_{n,\bullet-1}(\fancy{L}X)\] is then given by the cubical boundary map ignoring the fixed factor of $\Delta^n$.

The mixed complex $(C_\bullet^{\Delta \times \mathbb{I}}(\fancy{L}X), \partial, \mathcal{R})$ arises as the total complex of a cyclic dg $R$-module structure on the assignment
\[[n] \mapsto (C^{\Delta \times \mathbb{I}}_{n,\bullet}(\fancy{L}X), \partial^{\mathbb{I}}), \] as we now explain. For any
$\sigma: \Delta^n \times I^m \rightarrow \fancy{L}X$
and $\varepsilon \in S^1$, denote by $\sigma_\varepsilon \colon \Delta^n \times I^m \to \fancy{L}X$ the map given by
\[
\sigma_{\varepsilon}(u, v)(\theta) = \sigma(u, v)(\theta+ \epsilon) \]
on any $(u, v, \theta) \in \Delta^n \times I^m\times S^1$. In other words $\sigma_\varepsilon = \varepsilon \cdot \sigma,$ where $\cdot$ denotes the action of $S^1$ on $C_\bullet^{\Delta \times \mathbb{I}}(\fancy{L}X)$ induced by rotation of loops. Writing $(u_0,\dots,u_n, v)$ for the coordinates on $\Delta^n \times I^m$, the cyclic structure maps are given by
\begin{align*}
    &\partial_i(\sigma)(u_0,\dots, u_{n-1},v) = \sigma(u_0,\dots, u_{i-1} ,0,u_{i},\dots ,u_{n-1},v) \quad i=0,\dots, n\\
    & s_i(\sigma)(u_0, \dots ,u_{n+1},v) = \sigma(u_0, \dots, u_{i}+u_{i+1}, \dots, u_{n+1},v) \quad i=0,\dots, n\\
    & t(\sigma)(u_0,\dots , u_{n},v) = \sigma_{-u_0}(u_1,\dots, u_n,u_0,v).
\end{align*}
The above maps satisfy the cyclic compatibilies. The extra degeneracy is given explicitly by
\[
s(\sigma)(u_0,\dots , u_{n+1},v)=ts_n(\sigma)(u_0,\dots , u_{n+1},v)= \sigma_{-u_0}(u_1,\dots, u_n, u_{n+1}+ u_0,v).
\]
The \textit{normalized} total complex of the underlying simplicial dg $R$-module structure is exactly $C_\bullet^{\Delta \times \mathbb{I}} (\fancy{L}X)$. The induced Connes' operator \[Q \colon C_\bullet^{\Delta \times \mathbb{I}}(\fancy{L}X) \to C_{\bullet+1}^{\Delta \times \mathbb{I}}(\fancy{L}X)\] is defined by \[Q=s \circ ( \sum_{i=0}^n (-1)^n t)^i \colon C^{\Delta \times \mathbb{I}}_{n,\bullet}(\fancy{L}X) \to C^{\Delta \times \mathbb{I}}_{n+1,\bullet}(\fancy{L}X).\] Explicitly, on any $\sigma \colon \Delta^n \times I^m \to \fancy{L}X$, we have
\[Q(\sigma)(u_0,\ldots,u_{n+1},v)  =  \sum\limits_{i=0}^{n} (-1)^{ni} \sigma_{(-\sum_{j=0}^i u_j)}(u_{i+1},\dots, u_n,u_{n+1}+u_0,u_1\dots, u_i,v).\]
\begin{proposition}
The operators $\mathcal{R} \colon C_\bullet^{\Delta \times \mathbb{I}}(\fancy{L}X) \to C_{\bullet+1}^{\Delta \times \mathbb{I}}(\fancy{L}X) $ and $Q \colon C_\bullet^{\Delta \times \mathbb{I}}(\fancy{L}X) \to C_{\bullet+1}^{\Delta \times \mathbb{I}}(\fancy{L}X)$ agree. 
\end{proposition}
\begin{proof}
    Let $\sigma:\Delta^n \times I^m \rightarrow \fancy{L}X$, fix a point $(k_0,\ldots, k_n, v) \in \Delta^n \times I^m$, and denote
    $\sigma(k_0,\dots,k_n,v)$ by $ \gamma$.
    Since \[s(\sigma)(u_0,\dots, u_{n+1},v) = \sigma_{-u_0}(u_1,\dots, u_{n+1}+u_0,v),\] we have 
    $s(\sigma)(0,k_0, \ldots, k_n,v)= \gamma$. Furthermore, as long as 
    \begin{equation}\label{line}
    u_j=k_{j-1} \text{ for } j= 1,\dots , n \text{ , and }
    u_{n+1}+u_0 = k_n,
    \end{equation}
    holds, $s(\sigma)(u_0,\dots,u_{n+1},v)$ is in the $S^1$-orbit of $\gamma$. This means we may vary $u_0$ while staying in the $S^1$-orbit of $\gamma$. More precisely, if $u_{n+1}=k_n-u_0$, then 
    $ 0 \leq u_0 \leq k_{n}$; so that 
    the line segment determined by \ref{line} parameterizes $-\varepsilon \cdot \gamma = \sigma_{-\varepsilon}(k_0,\dots, k_n,v)$ for $0 \leq \varepsilon \leq k_n$. This part of the $S^1$-orbit of $\gamma$ is then determined by the term $s(\sigma)=st^0(\sigma)$ in the sum defining $Q(\sigma)$.
    
In general, $st^i(\sigma)$ restricted to the line segment
    \begin{align*}
        &u_j=k_{j-(i+1)} \text{ for } i<j\leq n\\
        &u_{n+1}+u_0 = k_{n-i}\\
        &u_j=k_{n-i+j} \text{ for } 0<j\leq i
    \end{align*}
    parameterizes $(-\varepsilon)\cdot \gamma = \sigma_{-\varepsilon}(k_0,\dots,k_n,v)$ for $\sum_{j=n-i+1}^{n}k_j \leq \varepsilon \leq \sum_{j=n-i}^{n}k_j$.  This part of the $S^1$-orbit of $\gamma$ is then determined by the term $st^i(\sigma)$ in the sum defining $Q(\sigma)$. From this description it follows that  \[st^i(\sigma)=(-1)^{ni} C_\bullet^{\Delta \times \mathbb{I}}(r) \Big((\iota \times \sigma)|_{\Delta^{n+1}_i}\Big),
    \] the $i$-th term in the sum defining $\mathcal{R}(\sigma)$.\end{proof}

\subsection{The (co)Hochschild complex as a mixed complex} \label{Hochasmixedcomplex}

The Hochschild complex  $(\text{Hoch}_{\bullet}(A),\delta)$ of a dg category $A$ may be endowed with a natural mixed complex structure as follows. Recall that \textit{Connes' operator} on the (normalized) Hochschild complex is the map
 $$B: \text{Hoch}_\bullet(A) \rightarrow \text{Hoch}_{\bullet + 1}(A)$$
defined by
$$B( [a_1 \vert \cdots \vert a_p]a_{p+1} ) = \sum\limits_{i=1}^{p+1} \pm [a_i \vert \cdots \vert a_{p+1} \vert a_1 \vert \cdots \vert a_{i-1}] \id_{s(a_i)}. $$
This map satisfies $B^2 = 0$ and $\delta B + B \delta = 0$. Consequently, $(\text{Hoch}_{\bullet}(A), \delta , B)$ is a mixed complex naturally associated to $A$. This mixed complex arises as the total complex of a cyclic dg $R$-module structure on the assignment

\[ [n] \mapsto \mathcal{Z}_n(A):= \big(\mathcal{M}(A) \underset{A_0}{\square} \overset{n}{\cdots} \underset{A_0}{\square} \mathcal{M}(A) \big) \underset{A_0 \otimes A_0^{\text{op}}}{\square} \mathcal{M}(A),\] 
where $A_0= R[\text{Obj}(A)]$ is the coalgebra generated by the objects of $A$ with diagonal coproduct and $\underset{K}{\square}$ denotes the cotensor product over a coalgebra $K$. In other words, $\mathcal{Z}_n$ is generated by composable strings $(a_1,\ldots,a_{n},a_{n+1})$ of morphisms in $A$ in which the target of $a_{n+1}$ is the source of $a_1$. The faces, degeneracies, and cyclic operators are defined as:
    \begin{align*}
    &d_i(a_1,\ldots,a_{n},a_{n+1}) = (a_1,\ldots, a_ia_{i+1}, \ldots, a_{n},a_{n+1}) \quad \text{for } i=1,\dots ,n \\
&d_{0}(a_1,\ldots,a_{n},a_{n+1}) = (a_2, \ldots, a_{n}, a_{n+1}a_1) \\   
&s_i(a_1,\ldots,a_{n},a_{n+1}) = (a_1, \ldots, a_i, \text{id}_{\mathsf{t}(a_{i})}, a_{i+1}, \ldots, a_{n+1}), \quad \text{for } i=1,\dots ,n
\\
&s_0(a_1,\ldots,a_{n},a_{n+1}) = (a_1, \ldots, a_{n+1},\text{id}_{\mathsf{t}(a_{n+1})})\\
&t(a_1,\ldots,a_{n},a_{n+1})= (a_{n+1},a_1 \ldots, a_n).
\end{align*}
Above we have written composition as concatenation, i.e. $ab= (b \circ a)$.


Dually, the coHochschild complex $(\text{coHoch}_{\bullet}(C),d)$ of a categorical coalgebra $C$ may also be endowed with a natural mixed complex structure as follows. Define a linear map of degree $+1$ \[P: \text{coHoch}_{\bullet}(C) \rightarrow \text{coHoch}_{\bullet +1}(C)\] by \[P(c_0 \{ c_1 \vert \cdots \vert c_p \})= \sum_{i=1}^p \pm \varepsilon(c_0) c_i \{ c_{i+1} \vert \cdots \vert c_p \vert c_1 \vert \cdots \vert c_{i-1} \},\] where $\varepsilon \colon C \to R$ is the counit of $C$. The signs are determined by the Koszul sign rule. This map satisfies $P^2 = 0$ and $d P + P d = 0$. Consequently, $(\text{coHoch}_{\bullet}(C), d , P)$ is a mixed complex naturally associated to $C$. This mixed complex arises by taking the \textit{direct sum} totalization of a co-cyclic dg $R$-module, but this perspective is not necessary for the present article.

\begin{proposition} \label{pimapofmixedcomplexes} For any categorical coalgebra $C$, the map $\overline{\pi} \colon \textnormal{Hoch}_\bullet( \mathbf{\Omega}C) \to \textnormal{coHoch}_{\bullet}(C)$ defined in \ref{overlinepi} is a natural quasi-isomorphism of mixed complexes.
\end{proposition}
\begin{proof}
    Let $a=[a_1 | \dots | a_p]a_{p+1}$ be a generator of $\textnormal{Hoch}_\bullet( \mathbf{\Omega}C)$. If $p>1$ then $P \overline{\pi} (a) = \overline{\pi} B(a) = 0 $. If $p=1$, then $a$ is of the form 
    \[
    a = [a_1]a_2 = [\{c_1| \dots | c_q\}]a_2 .
    \]
    If $q>0$, then 
    \[
    \overline{\pi} B(a)=\overline{\pi} ( [a_1 | a_2]\id_{s(a_1)} - [a_2 | a_1]\id_{s(a_2)} ) =0
    \] 
    On the other hand
    \[
    \overline{\pi} (a) = \sum\limits_{i=1}^q \pm c_i \{c_{i+1}| \dots | c_q \} a_2 \{ c_1 | \dots | c_{i-1}\} 
    \]
    where each $c_i$ has degree $>0$, so in particular $\varepsilon(c_i)=0$ for all $i$, and therefore $P\overline{\pi}(a)=0$. The only non-trivial case arises when $p= 0$, i.e. when $a$ is of the form
    \[
    a = [\id_x]\{k_1 | \dots | k_r\}
    \]
    for some $x \in C_0$. Then
    \begin{align*}
          P \overline{\pi} ([\id_x]\{k_1 | \dots | k_r\}) &=  P(x\{k_1 | \dots | k_r\})\\
        &=  \sum\limits_{i=1}^r \pm \varepsilon(x) k_i \{k_{i+1} | \dots | k_r | k_1 |\dots |k_{i-1} \} \\ 
        & = \overline{\pi} ([\id_x \{ k_1 | \dots | k_r \}]\id_{x})\\
        & = \overline{\pi} B ([\id_x]\{k_1 | \dots | k_r\} ).
    \end{align*}
\end{proof}
\begin{corollary}\label{invariance}
Any $\mathbf{\Omega}$-quasi-equivalence  $f \colon C \to C'$ of categorical coalgebras induces a quasi-isomorphism of mixed complexes $\textnormal{coHoch}_\bullet(f) \colon  \textnormal{coHoch}_\bullet(C) \to \textnormal{coHoch}_\bullet(C').$
\end{corollary}
\begin{proof}
This follows from Proposition \ref{pimapofmixedcomplexes} together with the naturality of the (co)Hochschild complex and the quasi-equivalence invariance of Hochschild homology for locally $R$-flat dg categories. 
\end{proof}
\begin{remark}
The map $\overline{\alpha} \colon \text{coHoch}_{\bullet}(C) \to \text{Hoch}_{\bullet}(\mathbf{\Omega}C)$, which is a chain homotopy inverse of $\overline{\pi}$, does not strictly intertwine the operators $P$ and $B$.  
\end{remark}

\begin{proposition} \label{Gmapofmixedcomplexes}
For any topological space $X$, the map 
 \[\mathfrak{G} \colon \textnormal{Hoch}_\bullet(\mathbf{\Omega}(\mathcal{C}(X))) \rightarrow C^{\Delta \times \mathbb{I}}_\bullet(\fancy{L}X) \] 
 defined in \ref{Gmaphoch} is a quasi-isomorphism of mixed complexes inducing a natural isomorphism on homology. 
\end{proposition}
\begin{proof}
The map $\mathfrak{G} \colon \text{Hoch}_\bullet(\mathbf{\Omega}(\mathcal{C}(X))) \rightarrow C^{\Delta \times \mathbb{I}}_\bullet(\fancy{L}X)$ is given as the composition
\[\textnormal{Hoch}_\bullet(\mathbf{\Omega}(\mathcal{C}(X))) \xrightarrow{\textnormal{Hoch}_\bullet(\Gamma)}  \textnormal{Hoch}_\bullet(\mathbf{P}^{\mathbb{I}}X) \xrightarrow{\lambda^{\mathbb{I}}}  C^{\Delta \times \mathbb{I}}_\bullet(\fancy{L}X), 
\]
where $\Gamma \colon \mathbf{\Omega}(\mathcal{C}(X)) \to \mathbf{P}^{\mathbb{I}}X$ is (the many object version of) Adams' classical map and $\lambda^{\mathbb{I}}$ is induced by (a many-object version of) the map of cyclic topological spaces $\lambda$ defined in \cite[Section V.1]{Goodwillie} upon taking cubical chains and passing to total complexes.  In fact, in the proof of Lemma V.1.3 in \cite{Goodwillie}, it is noticed that $\lambda$ is a weak homotopy equivalence of cyclic topological spaces. For a complete calculation see \cite{tolosa-thesis}. Hence, the induced map $\lambda^{\mathbb{I}}$ is a quasi-isomorphism of mixed complexes. The map $\Gamma \colon \mathbf{\Omega}(\mathcal{C}(X)) \to \mathbf{P}^{\mathbb{I}}X$ is a quasi-equivalence between two dg categories both of which have chain complexes of morphisms that are flat as graded $R$-modules, which is natural upon passing to homology on the morphisms. Therefore, by the quasi-equivalence invariance of the Hochschild chain complex, it follows that $\textnormal{Hoch}_\bullet(\Gamma)$ is a quasi-isomorphism of mixed complexes, which is natural on homology. 
\end{proof}
The following, which is Theorem \ref{thm1} in the introduction, is an immediate consequence of Propositions \ref{pimapofmixedcomplexes} and \ref{Gmapofmixedcomplexes}.
\begin{theorem} \label{maintheorem}
For any topological space $X$, there is a zig-zag of quasi-isomorphisms of mixed complexes
    \[
    (\textnormal{coHoch}_\bullet (\mathcal{C}(X)), d,P) \xleftarrow{\overline{\pi}} (\textnormal{Hoch}_\bullet( \mathbf{\Omega} \mathcal{C}(X)),\delta,B) \xrightarrow{\mathcal{E} \circ \mathfrak{G}} (C_\bullet(\fancy{L}X), \partial, \mathcal{R}).
    \] 
inducing natural isomorphisms on homology. 
\end{theorem}

\begin{remark} Corollary \ref{invariance} and Theorem \ref{maintheorem} imply that the coHochschild complex of any categorical coalgebra $\mathbf{\Omega}$-quasi-equivalent to $\mathcal{C}(X)$ models the singular chains on $\mathcal{L}X$ as a mixed complex. 
\end{remark}

\subsection{Cyclic homology and $S^1$-equivariant homology}
Given any mixed complex $(M,b,B)$, one may consider the graded $R$-module $M\llbracket u , u^{-1} \rrbracket$ defined as the power series ring on $u$ and $u^{-1}$, where $u$ is a formal variable of degree $2$, so that the grading on monomials is given by $\abs{mu^i}=\abs{m}+2i$. The following lemma describes a standard construction.
\begin{lemma}
If $(M,b,B)$ is a mixed complex, then $(M \llbracket u,u^{-1} \rrbracket, b+u^{-1}B)$ is a dg $R$-module. Furthermore, there is a natural isomorphism of dg $R$-modules \[M\llbracket u , u^{-1} \rrbracket \iso Tot^{\Pi}(\fancy{B}M),\] where $\fancy{B}M$ is the bi-complex associated to $(M,b,B)$ and $Tot^{\Pi}$ is the (direct product) totalization of a bi-complex, i.e. if $(C_{\bullet,\bullet}, d_{ver} ,d_{hor})$ is a bi-complex, then $Tot^{\Pi}(C)_n = \prod\limits_{p+q=n} C_{p,q}$.
\end{lemma}

There are eight variations of the above construction given by subcomplexes and quotients of $M\llbracket u,u^{-1} \rrbracket$ giving rise to eight versions of \textit{cyclic chains} associated to a mixed complex. Namely, one may consider
\[\begin{tikzcd}
	\textcolor{rgb,255:red,214;green,92;blue,92}{M\llbracket u,u^{-1} \rrbracket} && {M[u,u^{-1}]} &&
	{M\llbracket u,u^{-1} ]} && {M [u,u^{-1} \rrbracket} \\
	\textcolor{rgb,255:red,214;green,92;blue,92}{M\llbracket u \rrbracket} && \textcolor{rgb,255:red,214;green,92;blue,92}{M[u^{-1}]} && {M[u]}
	 && {M\llbracket u^{-1} \rrbracket}
\end{tikzcd}\]
with the induced differentials. The highlighted complexes are called \textit{classical versions} of cyclic chains: $M\llbracket u , u^{-1} \rrbracket$ is called the \textit{periodic cyclic chains}, $M[u^{-1}]$ the \textit{negative cyclic chains}, and $M\llbracket u \rrbracket$ the \textit{positive cyclic chains}. The complexes $M[u]$ and $M\llbracket u \rrbracket$ are quotients of $M\llbracket u , u^{-1} \rrbracket$, while the rest are subcomplexes of $M\llbracket u , u^{-1} \rrbracket$. All these versions are non-isomorphic in general. 

We will denote by $M^+ , M^-$, and $M^{\infty}$, the \textit{homology} of the positive, negative, and periodic cyclic chain complexes of the mixed complex $M$, respectively. If the mixed complex $(M,b,B)$ arises as the total complex of a cyclic dg $R$-module (see \ref{mixedcomplexes}) all of the eight versions above of the cyclic chains are either acyclic or quasi-isomorphic to one of the classical versions \cite{cieliebakvolkov}. For any categorical coalgebra $C$, we apply these constructions to the mixed complex $(\textnormal{coHoch}_\bullet(C), d, P)$.

\begin{definition}\label{3versions} Let $C$ be an $R$-categorical coalgebra. We call the graded $R$-modules $\textnormal{coHoch}_\bullet(C)^+$, $\textnormal{coHoch}_\bullet(C)^-$, and $\textnormal{coHoch}_\bullet(C)^{\infty}$ the \textit{positive, negative and periodic cyclic homology of} $C$, respectively. 
\end{definition}

Given a topological space $X$, the positive cyclic homology of $(C_\bullet(\fancy{L}X),\partial,\mathcal{R})$, and consequently the positive cyclic homology of $(C^{\Delta \times \mathbb{I}}_\bullet(\fancy{L}X),\partial,\mathcal{R})$, is naturally isomorphic to $H^{S^1}_\bullet(\fancy{L}X)$, the $S^1$-equivariant homology of the free loop space as an $S^1$-space \cite{jones1987cyclic}. The latter is defined as the singular homology of its Borel construction, i.e. \[H^{S^1}_\bullet(\fancy{L}X) = H_\bullet(ES^1 \times_{S^1} \fancy{L}X).\] The mixed complex $(C^{\Delta \times \mathbb{I}}_\bullet(\fancy{L}X),\partial,\mathcal{R})$ is the mixed complex of a cyclic dg $R$-module as discussed in \ref{chainsonLX}. In particular, there are only three versions of cyclic homology of $(C^{\Delta \times \mathbb{I}}_\bullet(\fancy{L}X),\partial,\mathcal{R})$, up to isomorphism.

\begin{theorem}
   For any topological space $X$, there are natural isomorphisms
    \begin{align}
        &\textnormal{coHoch}_\bullet(\mathcal{C}(X))^+ \iso C_\bullet(\fancy{L}X) ^+ \\
        & \textnormal{coHoch}_\bullet(\mathcal{C}(X))^- \iso  C_\bullet(\fancy{L}X)^-,\\
        & \textnormal{coHoch}_\bullet(\mathcal{C}(X))^{\infty} \iso C_\bullet(\fancy{L}X)^{\infty} .
    \end{align}
    In particular, the first isomorphism above gives rise to a natural isomorphism between the positive cyclic homology of the categorical coalgebra $\mathcal{C}(X)$ of singular chains on $X$ and the $S^1$-equivariant homology of $\mathcal{L}X$.
\end{theorem}

\begin{proof}
This follows directly from Theorem \ref{maintheorem} since quasi-isomorphisms of mixed complexes induce quasi-isomorphisms of dg $R$-modules after passing to any of the classical versions of the cyclic chain complexes. Furthermore, the positive cyclic homology of the mixed complex $(C_\bullet (\mathcal{L}X), \partial, \mathcal{R})$ is the $S^1$-equivariant homology of $\mathcal{L}X$. 
\end{proof}
\newpage
\printbibliography

\end{document}